\documentclass[11pt]{article}
\usepackage{amsmath,amssymb,amsfonts,amsxtra}
\usepackage{amsthm, bm}
\usepackage{graphicx,psfrag}
\usepackage{booktabs}
\usepackage{subfigure}
\usepackage{multirow}
\usepackage{url}
\usepackage[usenames]{color}
\usepackage[colorlinks,linktocpage,linkcolor=blue]{hyperref}
\usepackage[margin=1.1in]{geometry}
\allowdisplaybreaks

\definecolor{darkred}{rgb}{.7,0,0}

\definecolor{green}{rgb}{0,0.7,0}

\newtheoremstyle{thmm}{1.5ex plus 1ex minus .2ex}{1.5ex plus 1ex minus.2ex}{\rmfamily}{}{\bfseries}{}{1em}{} \theoremstyle{thmm}
\newtheorem{theorem}{Theorem}[section]
\newtheorem{lemma}{Lemma}[section]
\newtheorem{assumption}{Assumption}[section]

\newtheorem{remark}{Remark}[section]
\renewenvironment{proof}[1][Proof]{\noindent\textit{#1. }
}{\hfill$\square$}

\def\R{\mathbb{R}}

\def\d{{\mathrm d}}
\def\Omega{{\varOmega}}

\definecolor{mygray}{rgb}{0.9,0.9,0.9}

\definecolor{darkred}{rgb}{0.55,0,0.0}

\DeclareMathOperator*{\esssup}{ess\,sup}

\title{\Large\bf Error estimates for the scalar auxiliary variable (SAV) scheme to the Cahn-Hilliard equation}
\date{}

\graphicspath{{images/}}	
\begin{document}


\author{\normalsize 
Shu Ma\thanks{Department of Mathematics, City University of Hong Kong, 83 Tat Chee Avenue, Kowloon, Hong Kong, P.R. China.
Email address: shuma2@cityu.edu.hk.}
\and 
Weifeng Qiu\thanks{Department of Mathematics, City University of Hong Kong, 83 Tat Chee Avenue, Kowloon, Hong Kong, P.R. China. 
Email address: weifeqiu@cityu.edu.hk.
}
\and
Xiaofeng Yang\thanks{Department of Mathematics, University of South Carolina, Columbia, SC, 29208, USA. Email address: xfyang@math.sc.edu.}
}

\date{}




\maketitle
\vspace{-10pt}

\begin{abstract} The optimal error estimate that depending only on the polynomial degree of $ \varepsilon^{-1}$ is established for the temporal semi-discrete scheme of the Cahn-Hilliard equation, which is based on the scalar auxiliary variable (SAV) formulation.
The key to our analysis is to convert the structure of the SAV time-stepping scheme back to a form compatible with the original format of the Cahn-Hilliard equation,  which makes it feasible to use spectral estimates to handle the nonlinear term.
Based on the transformation of the SAV numerical scheme, the optimal error estimate for the temporal semi-discrete scheme which depends only on the low polynomial  order of $\varepsilon^{-1}$ instead of the exponential order, is derived by using mathematical induction, spectral arguments, and the superconvergence properties of some nonlinear terms.
Numerical examples are provided to illustrate the discrete energy decay property and validate our theoretical convergence analysis.


\end{abstract}



\section{Introduction}\label{sec:intr}
In this paper we consider the initial boundary value problem for the  Cahn-Hilliard (CH) phase field equation
\begin{subequations}\label{pde-CH}
\begin{alignat}{2} 
\partial_t u & = \Delta \big( - \varepsilon \Delta u + \frac{1}{\varepsilon} f(u) \big) &&\qquad \mbox{in}\,\,\,\Omega\times(0,T], \label{pde-CH-u}\\ 
\partial_{\bm n} u & = \partial_{\bm n} \big ( - \varepsilon \Delta u + \frac{1}{\varepsilon} f(u) \big) = 0 && \qquad \mbox{on}\,\,\, \partial \Omega\times(0,T],\\
u(\cdot, 0) & = u_0, &&\qquad\mbox{in}\,\,\,\Omega, 
\end{alignat} 
\end{subequations} 
where $\Omega \subset R^d$, $d = 2, 3$ is a bounded domain, ${\bm n}$ is the
outward normal, $\varepsilon$ is a small parameter,  and $f$ is the derivative of a non-negative potential function $F$ with two local minima, i.e., $f = F'$. For instance, the Ginzburg-Landau energy function
$$
F(v) = \frac14 (v^2 - 1)^2 \quad \mbox{and} \quad f(v) = v^3 - v. 
$$
In view of its wide application as a phase field model \cite{cahn1958free, yue2004diffuse}, many numerical methods and analyses have been developed for approximating the Cahn-Hilliard equation \eqref{pde-CH}. On the one hand, most of these works have been performed for the Cahn-Hilliard equation with  fixed $\varepsilon>0$.
As pointed out in \cite{barrett1995error, du1991numerical, elliott1989nonconforming, elliott1989second}, error estimates using the direct Gronwall inequality argument yield a constant factor $e^{T/\varepsilon}$, and as a result the error grows exponentially as $\varepsilon \to 0$. Such an estimate is clearly not useful for very small $\varepsilon$, especially in solving the problem of whether the computed numerical interface converges to the original sharp interface of the Hele-Shaw problem when $\varepsilon \to 0$, see \cite{chen1996global, pego1989front} for details. 
To overcome this difficulty, Feng and Prohl \cite{feng2004error} first established a priori error estimates with polynomial dependence on $\varepsilon^{-1}$ for the time-discrete format of the Cahn-Hilliard equation.
Then, Feng and Wu obtained a posteriori error estimates with polynomial-order of $\varepsilon^{-1}$ in 
\cite{feng2008posteriori}  for the same time-discrete methods.
The main idea of the polynomial-order error estimates of $\varepsilon^{-1}$ is to use the spectral estimates given by Alikakos and Fusco \cite{alikakos1993spectrum} and Chen \cite{chen1994spectrum} for the linearized Cahn-Hilliard operator to handle the nonlinear term in the error analysis. 
After this, spectral estimates were frequently used to eliminate the exponential dependence on  $\varepsilon^{-1}$ in the error analyses of other numerical methods for the  Cahn-Hilliard  equations (see \cite{bartels2011error, feng2007fully, feng2016analysis, prohl2005numerical, wang2018efficient} and references therein) and the related phase field equations, including the Allen-Cahn equations in \cite{akrivis2022error, bartels2011error, bartels2011quasi, bartels2011robust, feng2015analysis, feng2003numerical, feng2004analysis,  feng2005posteriori}, the Ginzburg-Landau equations in \cite{bartels2005robust}, and the phase field models with nonlinear constitutive laws in \cite{eck2010error}.

On the other hand, it is wellknow that the Cahn-Hilliard equation \eqref{pde-CH} is the $H^{-1}$-gradient flow of the energy functional 
\begin{align}
E[u] := \int_{\Omega} \Big( \frac{\varepsilon}{2} |\nabla u |^2 + \frac{1}{\varepsilon} F(u) \Big) \, \d x.
\end{align}
As a result, the solution of the Cahn-Hilliard equation has decaying energy. Indeed, testing \eqref{pde-CH} 
by $ - \varepsilon \Delta u + \frac{1}{\varepsilon} f(u)$ yields
\begin{align}
E[u(\cdot, t_2)] - E[u(\cdot, t_1)] = - \int_{t_1}^{t_2} \|\partial_t u \|_{H^{-1}}^2 \, \d t \le 0,
\end{align}
where $0 \le t_1 \le t_2 \le T$ and $\|\partial_t u\|_{H^{-1}} = \|\nabla w\|$ with $w = - \varepsilon \Delta u + \frac{1}{\varepsilon} f(u)$.

Accordingly, great efforts have been devoted to the construction of efficient and accurate numerical methods that preserve the energy decay properties at the discrete level. In particular, for those widely used linear time-stepping methods with energy decay, including stabilized semi-implicit schemes \cite{cai2017error, cai2018error, shen2010numerical}, invariant energy quadratization (IEQ) methods \cite{yang2016linear, yang2017efficient, yang2017linear}, and the scalar auxiliary variable (SAV) approach in \cite{shen2018convergence, shen2018scalar} and \cite{akrivis2019energy}. 
Optimal error estimates of these time-stepping schemes for the Cahn-Hilliard equations with fixed $\varepsilon^{-1}$ are now well developed. As far as the error analysis is concerned, the main difficulty that remains with these methods is how to establish error bounds that depend only on the polynomial order of $\varepsilon^{-1}$ rather than the exponential order for small $\varepsilon \to 0$.
This is due to the fact that compared to Feng's previous work \cite{feng2007fully,feng2016analysis, feng2004error, feng2008posteriori, prohl2005numerical}, the numerical methods based on the IEQ/SAV formulation break the standard structure of the nonlinear term of the Cahn-Hilliard equation, which is crucial to the utilization of the spectral arguments. This makes the spectral estimates ineffective in estimating errors for the IEQ/SAV approach.
To the best of our knowledge, the optimal error estimates depending only on the polynomial order  of $\varepsilon^{-1}$ for the IEQ/SAV methods to the Cahn-Hilliard equation remains open.

The objective of this paper is to establish error bounds which depend on $\varepsilon^{-1}$ only in low polynomial order for a semi-discrete methods based on the SAV formulation of the Cahn-Hilliard equation. The SAV reformulation of the Cahn-Hilliard equation was introduced in \cite{shen2018scalar,shen2019new} as an enhanced version of the invariant energy quadratization (IEQ) approach \cite{yang2016linear, yang2017efficient, yang2017linear, yang2017numerical}, for developing energy-decay methods at the discrete level. By reconstructing the system based on the SAV reformulation, we obtain a time semi-discrete scheme, which is linear and easy-to-implement.
The SAV formulation introduces new difficulties to the error analysis for the Cahn-Hilliard equation due to the presence of a new scalar $r$ in the nonlinear part (see equation \eqref{SAV-CH-w}), which  alters the structure of the original Cahn-Hilliard equation and makes the spectral argument not directly applicable.
To improve the current error analysis, Zhang and Yang \cite{yang2022error} have recently made a breakthrough in the estimates of the IEQ method for the Allen-Cahn equation. They established optimal error bounds on the polynomial dependence of $\varepsilon^{-1}$ for the IEQ-based numerical schemes.
Inspired by \cite{yang2022error}, we rewrite the format of the SAV scheme using the new scalar variable into a form compatible with the original format of the Cahn-Hilliard equation,  which makes it feasible to use spectral estimates in the error analysis.
However, this structural transformation will accordingly introduces a strong perturbation term that needs to be delicately controlled. Furthermore, unlike the Allen-Cahn equation which is a gradient flow in $L^2(\Omega)$, the Cahn-Hilliard equation is a gradient flow in $H^{-1}(\Omega)$, which makes the analysis for the Cahn-Hilliard equation in this paper more delicate and complicated than that for the Allen-Cahn equation. In our analysis, these difficulties are overcome by combining the following techniques:
\begin{itemize}
\item[(1)] To use the spectral estimates of the linearized Cahn-Hilliard operator to deal with the nonlinear potential term in the error analysis, we reconvert the structure of the SAV scheme into a form compatible with the original Cahn-Hilliard equation \eqref{pde-CH}.

\item[(2)] An inductive argument is used to deal with the difficulties caused by the strong perturbation term that appear in the structural transformation (see equation \eqref{err-eqn}). In particular, error bounds of $\|\nabla e^{i-1}- \nabla e^{i-2}\|$ and $\|e^{i-1}- e^{i-2}\|_{H^{-1}}$ for $i\le n$ need to be established. By integrating them into the estimates of the perturbation terms, the super-convergence characteristics of some of their resulting nonlinear terms will complete the mathematical induction method.

\item[(3)] Given that the solution of the Cahn-Hilliard problem \eqref{pde-CH} preserves the total mass property (i.e. $\frac{\d }{\d t} \int_{\Omega} u(x,t) \, \d x = 0$), which is not possessed by the corresponding Allen-Cahn problem, we can establish its optimal error bounds in the $L^{\infty}(0, T; H^{-1}(\Omega))$-norm with polynomial dependence on $\varepsilon^{-1}$, with the help of an error estimate in the $L^2(0, T; H^1(\Omega))$-norm. This is different from the error estimates in the $L^{\infty}(0, T; L^2(\Omega))$-norm of the Allen-Cahn equation given in \cite{yang2022error}.
\end{itemize}
To the best of our knowledge, this is the first error estimate of the polynomial dependence on $\varepsilon^{-1}$ for the IEQ/SAV-type schemes of the Cahn-Hilliard equation \eqref{pde-CH}.


The rest of this paper is organized as follows. 
In Section \ref{sec-2}, we present the SAV reformulation of the Cahn-Hilliard equation and introduce an equivalent transformation of the SAV time-stepping scheme. 
In Section \ref{sec-3}, we show the properties of energy decay and derive the consistency estimates for the proposed method.
In Section \ref{sec-4}, we present an error estimate of the semi-discrete SAV scheme to derive a convergence rate that does not depend on $\varepsilon^{-1}$ exponentially. The spectrum estimate plays a crucial role in the proof. Finally, in Section \ref{sec-5}, we present a few numerical experiments to validate the theoretical results.

\section{Formulation of the Semi-discrete SAV scheme}\label{sec-2}

In this section, we construct a backward Euler implicit-explicit type temporal semi-discrete numerical scheme based on the SAV reformulation of the CH equation \eqref{pde-CH}, and also present an equivalent formulation  of the SAV scheme.




\subsection{Function spaces}
Let $W^{s, p}(\Omega)$ denote the usual Sobolev spaces, and $H^s(\Omega)$ denote the Hilbert spaces 
$W^{s, 2}(\Omega)$ with norm $\|\cdot\|_{H^s}$.
Let $\|\cdot\|$ and $(\cdot, \cdot)$ represent the $L^2$ norm and $L^2$  inner product, respectively. In addition, define for $p \ge 0$
\begin{align}
H^{-p}(\Omega):=(H^p(\Omega))^*, \quad H_0^{-p}(\Omega):= \{ u \in H^{-p}(\Omega) \,| \,\langle u, 1\rangle_p = 0 \}, 
\end{align}
where $\langle \cdot, \cdot\rangle_p$ stands for the dual product between $H^p(\Omega)$ and $H^{-p}(\Omega)$. 
We denote $L_0^2(\Omega) := H_0^0(\Omega)$. 
For $v \in L_0^2(\Omega)$, let $- \Delta^{-1} v:=v_1 \in H^1(\Omega) \cap L_0^2(\Omega)$, where $v_1$ is the solution to
\begin{align}
-\Delta v_1 = v \, \, \mbox{in} \, \, \Omega, \qquad  \partial_{\bm n} v_1 = 0 \,\, \mbox{on} \,\, \partial \Omega,
\end{align}
and $\|v\|_{-1} := \sqrt{(v, - \Delta^{-1} v)}$.

For  $v \in L_0^2(\Omega) \cap H^1(\Omega)$, we have the following inequality
\begin{align}\label{inq-l2tohm1}
\|v\|^2 = (\nabla v, \nabla (- \Delta)^{-1} v) \le \|\nabla v\| \|v\|_{-1}.
\end{align}

We denote by $C$ generic constant and $C_i$, $\Tilde C_i$, $\Tilde C$, $\kappa_i'$ and $\kappa_i$ specific constants, which are independent of $\tau$, $h$ and $\varepsilon$, but may possibly depend on the domain $\Omega$, $T$ and the constants of Sobolev inequalities.
We use notation $\lesssim$ in the sense that $f \lesssim g$ means that $f \le Cg$ with positive constant $C$ independent of $\tau$, $h$ and $\varepsilon$.

\subsection{The SAV reformulation}

The SAV formulation of the CH equation (cf. \cite{shen2018convergence, shen2018scalar}) introduces a scalar auxiliary variable   
\begin{align}
r = \sqrt{ \mbox{$\int_\Omega $} F(u) \d x + c_0}
\quad\mbox{with}\quad
g(u)=\frac{f(u)}{ \sqrt{ \mbox{$\int_\Omega $} F(u) \d x +c_0}},
\end{align}
with a positive $c_0$ (which guarantees that the function $r$ has a positive lower bound), and reformulate \eqref{pde-CH} as 
\begin{subequations}\label{SAV-CH} 
\begin{alignat}{2} 
\partial_t u  & = \Delta w  &&\qquad \mbox{in}\,\,\,\Omega\times(0,T], \label{SAV-CH-u}
\\
w &= - \varepsilon \Delta u + \frac{1}{\varepsilon} r g(u)  &&\qquad \mbox{in}\,\,\,\Omega\times(0,T], \label{SAV-CH-w}
\\ 
\frac{\d r}{\d t} &= \frac12 \big(g(u), \partial_tu\big) 
&&\qquad \mbox{in}\,\,\,\Omega\times(0,T], \label{SAV-CH-r}\\ 
\partial_{\bm n} u & = \partial_{\bm n} w = 0 && \qquad \mbox{on}\,\,\, \partial \Omega\times(0,T],
\\
u(\cdot, 0) & = u_0 &&\qquad\mbox{in}\,\,\,\Omega, \\
r(0) & = \sqrt{ \mbox{$\int_\Omega $} F(u_0) \d x+c_0}.&&
\end{alignat} 
\end{subequations} 
We define an energy functional with respect to $u$ and $r$:
\begin{align}
E(u, r) = \frac{\varepsilon}{2} \|\nabla u\|^2 + \frac{1}{\varepsilon} r^2,
\end{align}
and taking the $L^2$ inner product of the first equation \eqref{SAV-CH-u} with $w$, of the second equation \eqref{SAV-CH-w} with $\partial_t u$, and of the third equation \eqref{SAV-CH-r}
with $\frac{1}{\varepsilon} 2r$, performing integration by parts and summing up the two obtained equations, we find that
\begin{align}
\frac{\d}{\d t} E(u, r)  = - \|\nabla w\|^2 = - \|\partial_t u\|_{-1}^2 \leq 0.
\end{align}

\subsection{The equivalent formulation of the SAV scheme}\label{SAV-Euler-Reformulation} 
Let $\{t_n\}_{n=0}^{N+1}$ be a uniform partition of $[0,T]$ with the time step size $\tau$, where $N$ is a positive integer and hence $\tau=\frac{T}{N+1}$.
We consider the following temporal semi-discrete SAV scheme for solving the system \eqref{SAV-CH}:
\begin{equation}\label{SAV-Euler-v1} 
\left \{
\begin{aligned}
\frac{u^{n+1} - u^n}{\tau} & = \Delta w^{n + 1}, 
\\
w^{n+1}& = - \varepsilon \Delta u^{n+1} + \frac{1}{\varepsilon} r^{n+1} g(u^n),  \\ 
r^{n+1} - r^n & = \frac12 \big(g(u^n), u^{n+1} - u^n \big),   \\
\partial_{\bm n} u^{n+1} |_{\partial \Omega} &= \partial_{\bm n} w^{n+1} |_{\partial \Omega} = 0, 
\end{aligned} 
\right .
\end{equation} 
with $u^0 = u_0$ and $r^0 = r(0)$ for $n = 0, 1, \dots, N$.

By taking the $L^2$ inner product of the first equation in \eqref{SAV-Euler-v1} with $v = 1$, we have
the following conservation property, which  is important to the error estimates. 
\begin{lemma}
The numerical solution of \eqref{SAV-Euler-v1} satisfies
\begin{align}\label{lem:conserved-u}
\frac{1}{|\Omega|} \int_{\Omega} u^n(x) \, \d x = \frac{1}{|\Omega|} \int_{\Omega} u^0(x) \, \d x, \qquad n = 1, \dots, N,
\end{align}
and the error function $e^n := u(t_n) - u^n$ satisfies 
\begin{align}\label{lem:conserved-e}
\int_{\Omega} e^n \, \d x = 0, \qquad \qquad n = 1, \dots, N.
\end{align}
\end{lemma}
Because of \eqref{lem:conserved-u}-\eqref{lem:conserved-e}, $u^{n+1} - u^n$ and $e^n$ belong to $L_0^2(\Omega)$ such that we can define their $\|\cdot\|_{-1}$ norm. 
Also, the semi-discrete SAV scheme \eqref{SAV-Euler-v1} can be written as
\begin{equation}\label{SAV-Euler} \left \{
\begin{aligned} 
\Delta^{-1} \frac{u^{n+1} - u^n}{\tau} & = - \varepsilon \Delta u^{n+1} + \frac{1}{\varepsilon} r^{n+1} g(u^n),  \\
r^{n+1} - r^n & = \frac12 \big(g(u^n), u^{n+1} - u^n \big), \qquad n = 0, 1, \dots, N.  
\end{aligned} 
\right .
\end{equation} 



In order to avoid the  exponentially dependence of the error bound on $\frac{1}{\varepsilon}$ induced by using the Gronwall inequality, we need to use a spectral estimate of the linearized Cahn-Hilliard operator, which is given in \cite{alikakos1993spectrum, chen1994spectrum, feng2004error} and will be described in Section~\ref{sec-4}.

However, compared with the previous work \cite{feng2004error}, the SAV method \eqref{SAV-Euler-v1} alters the structure of the CH equation such that the spectral argument can not be applied directly. To achieve the ideal error bound, we need to transform the structure into a form compatible with the CH equation \eqref{pde-CH} so that the spectral estimate of the linearized Cahn-Hilliard operator can be used.
To this end, we define $A(v) = \sqrt{\int_{\Omega} F(v) \, \d x + c_0}$, the Gateaux derivatives of $A(v)$ can be 
defined as follows:
\begin{align}
D A(v, w)   & : = \frac12 \Big(\frac{f(v)}{\sqrt{\int_{\Omega} F(v) \, \d x + c_0}}, w \Big) 
= \frac12 \big( g(v), w\big),\\
D^2 A(v, w) & : = \frac12 \frac{\int_{\Omega} f'(v) w^2 \, \d x}{\sqrt{\int_{\Omega} F(v) \, \d x + c_0}} 
- \frac14 \frac{ \big(\int_{\Omega} f(v) w\, \d x \big)^2}{\Big(\int_{\Omega} F(v) \, \d x + c_0\Big)^{\frac32}}.
\end{align}
By Taylor expansion, we derive 
\begin{align}
A(u^i) = A(u^{i-1}) + \frac12 \big( g(u^{i - 1}), u^i - u^{i-1}\big) + \frac12 D^2A(\xi_i\,; u^i - u^{i-1}),
\end{align}
where $\xi_i = \theta u^i + (1 - \theta) u^{i-1}$ with $\theta \in (0, 1)$.
Thus we get
\begin{align*}
\frac12 \big(g(u^{i-1}), u^i - u^{i - 1}\big) = A(u^i) - A(u^{i-1}) -  \frac12 D^2 A (\xi_i\,; u^i - u^{i-1}),
\end{align*}
which together with the second equation in \eqref{SAV-Euler} implies
\begin{align*}
r^i - r^{i - 1} = A(u^i) - A(u^{i-1}) -  \frac12 D^2 A (\xi_i\,; u^i - u^{i-1}).
\end{align*}
After summing up the above equation from $i = 1$ to $n$, we derive
\begin{align*}
r^n - r^0 = A(u^n) - A(u^0) -  \frac12 \sum_{i = 1}^n D^2 A (\xi_i\,; u^i - u^{i-1}).
\end{align*}
Since $r^0 = A(u^0)$, we obtain
\begin{align}\label{sav-refor-rn}
r^n = A(u^n) -  \frac12 \sum_{i = 1}^n D^2 A (\xi_i\,; u^i - u^{i-1}).
\end{align}
Then the SAV scheme \eqref{SAV-Euler} can be written as 
\begin{align}
\Delta^{-1} \frac{u^{n+1} - u^n}{\tau} 
=\, & - \varepsilon \Delta u^{n+1} + \frac{1}{\varepsilon} r^n g(u^n) + \frac{1}{\varepsilon} g(u^n) (r^{n+1} - r^n) \\
=\, & - \varepsilon \Delta u^{n+1} + \frac{1}{\varepsilon} r^n g(u^n) + \frac{1}{2\varepsilon} g(u^n) \big(g(u^n) , u^{n+1} - u^n\big), \notag
\end{align} 
which together with \eqref{sav-refor-rn}  gives
\begin{align}\label{sav-refor-scheme}
\Delta^{-1} \frac{u^{n+1} - u^n}{\tau} 
=\, & - \varepsilon \Delta u^{n+1} + \frac{1}{\varepsilon} f(u^n) - \frac{1}{2\varepsilon} g(u^n) \sum_{i = 1}^n D^2 A(\xi_i\,; u^i - u^{i - 1})
\\
& + \frac{1}{2\varepsilon} g(u^n) \big(g(u^n), u^{n+1} - u^n\big). \notag
\end{align}
The above equation \eqref{sav-refor-scheme} provides an equivalent formulation of the semi-discrete SAV scheme \eqref{SAV-Euler-v1}, which  will be frequently used in the subsequent error analysis.

\section{Energy decay and consistency analysis}\label{sec-3}
In this section, we present several inequalities related to the proposed numerical method. 

\subsection{Assumption and regularity}
Before presenting the detailed numerical analysis, we first make some assumptions. The quartic growth of the Ginzburg-Landau energy function $F(u) = \frac14 (u^2 - 1)^2$ at infinity poses various technical difficulties for the analysis and approximation of CH equations.
Although the CH equation does not satisfy the maximum principle, if the maximum norm of the initial condition $u^0$ is bounded, it has been shown in \cite{caffarelli1995bound} that the maximum norm of the solution of the CH equation for the truncation potential $F(u)$ with quadratic growth rate at infinity is bounded. Therefore, it has been a common practice (cf. \cite{shen2010numerical}) to consider the CH equations with a truncated $F(u)$.
\begin{assumption} \label{ass-f}
We assume that the potential function $F(u)$ whose derivative $f(u) = F'(u)$ satisfies the following condition:
\begin{itemize}
\item[(i)]  $F \in C^4(\R)$, $F(\pm 1) = 0$, and $F >0 $ elsewhere.
\item[(ii)]  $f(\pm 1) = 0$, $f'(\pm 1) > 0$, and there exists a non-negative constant $L$ such that 
\begin{align}\label{ass-f1}
\max_{v \in \R}|f(v)| \le  L, \quad \max_{v \in \R}|f'(v)| \le  L \quad \mbox{and} \quad \max_{v \in \R}|f''(v)| \le  L.
\end{align}
\end{itemize}
\end{assumption}

In order to trace the dependence of the solution on the small parameter $\varepsilon > 0$, we assume that the solution of \eqref{pde-CH} satisfies the following conditions:
\begin{assumption}\label{ass-u0}
Suppose there exist positive $\varepsilon$-independent constants $m_0$ and $\rho_j$ for $j = 1, 2, 3$ such that the solution of \eqref{pde-CH} satisfies
%
\begin{subequations}
\begin{alignat}{2} 
\label{sta-u-a}
&  \frac{1}{|\Omega|} \int_{\Omega} u\, \d x = m_0 \in (-1, 1), 
\\
\label{sta-u-b}
& \esssup_{t \in [0, \infty]} \Big\{ \frac{\varepsilon}{2} \|\nabla u\|^2 + \frac{1}{\varepsilon} \int_{\Omega} F(u) \, \d x  \Big\} 
+ \int_0^\infty \|u_t\|_{-1}^2 \, \d s
\lesssim \varepsilon^{-\rho_1}, 
\\
\label{sta-u-d}
& \int_0^\infty \|u_t\|^2 \, \d s \lesssim \varepsilon^{- \rho_2}, \\
\label{sta-u-e}
& \int_0^\infty \|\Delta^{-1} u_{tt}\|_{-1}^2 \, \d t \lesssim \varepsilon^{-\rho_3}.
\end{alignat}
\end{subequations}
\end{assumption}
\begin{remark}
(a) Note that the conditions (i) and (ii)  are satisfied by
restricting the growth of  $F(v)$ for $|v| \ge M $.  
More precisely, 
for a  given $M \ge 1$, we can replace $F(v) = \frac14 (v^2 - 1)^2$ by a cut-off function $\hat F(v) \in C^4(\R)$ as follows:
\begin{equation} \label{truncate-F}
\hat F(v) = 
\left \{
\begin{aligned}
& ((2M)^2 - 1)2M(v - 2M) + \frac14 ((2M)^2 - 1)^2 && \mbox{for}\,\, v > 2M, \\
& \Phi_+(v) && \mbox{for}\,\, v \in (M, 2M], \\
& \frac14 (v^2 - 1)^2 && \mbox{for}\, \,v \in [-M, M], \\
& \Phi_-(v) && \mbox{for}\,\, v \in [-2M, -M), \\
& -((2M)^2 - 1)2M(v + 2M) + \frac14 ((2M)^2 - 1)^2 && \mbox{for}\, \,v < -2M,
\end{aligned}
\right .
\end{equation}
where $\Phi_+(v)$ and $\Phi_-(v) >0$ elsewhere between $M<|v|< 2M$ and satisfy
the required conditions at $|v| = M$ and $|v| = 2M$, respectively.
Then we replace $f(v) = (v^2 - 1)v$ by $\hat F'(v)$ which is 
\begin{equation}
\hat f(v) = \hat F'(v) = 
\left \{
\begin{aligned}
& ((2M)^2 - 1)2M  && \mbox{for}\,\, v > 2M, \\
& \Phi'_+(v) && \mbox{for}\,\, v \in (M, 2M], \\
& (v^2 - 1)v      && \mbox{for}\, \,v \in [-M, M], \\
& \Phi'_-(v) && \mbox{for}\,\, v \in [-2M, -M), \\
& -((2M)^2 - 1)2M && \mbox{for}\, \,v < -2M.
\end{aligned}
\right .
\end{equation}
In simplicity, we still denote the modified function $\hat F$ by $F$. It is then obvious that there exists $L$ such that \eqref{ass-f1} are satisfied with $f$ replaced by $\hat f$.

(b) 
The transformed SAV scheme \eqref{sav-refor-scheme} introduced a complicated term.
With the condition (ii) in the Assumption~\ref{ass-f}, we can get
\begin{align}
D^2 A(v; w) \lesssim \big(\|f'(v)\|_{L^\infty} + \|f(v)\|^2 \big) \|w\|^2 
\lesssim \|w\|^2,
\end{align}
which will be frequently used to control the difficult term in the error analysis.

(c) Assumption~\ref{ass-u0} can be achieved in many cases. For example,
suppose that $f$ satisfies Assumption \ref{ass-f},  $\partial \Omega$ is of 
class $C^{2, 1}$, $u^0 \in H^3(\Omega)$, and there exist positive $\varepsilon$-independent constants $\sigma_j$ for $j = 1, 2, 3$ such that
\begin{subequations}
\begin{alignat}{2} 
& E[u^0]=\frac{\varepsilon}{2}\|\nabla u^0\|^2 + \frac{1}{\varepsilon} \int_{\Omega} F(u^0)\, \d x \lesssim \varepsilon^{-2\sigma_1},\\
& \|-\varepsilon \Delta u^0 + \frac{1}{\varepsilon} f(u^0)\| \lesssim \varepsilon^{-2\sigma_2},\\
& \|-\varepsilon \Delta u^0 + \frac{1}{\varepsilon} f(u^0)\|_{H^1} \lesssim \varepsilon^{-2\sigma_3}.
\end{alignat}
\end{subequations}
Then the estimates \eqref{sta-u-a}--\eqref{sta-u-e} can be derived 
by standard test function techniques and satisfy: 
\begin{align*}
\rho_1 = 2\sigma_1, \quad 
\rho_2 = \max \{2\sigma_1 + 2, 2\sigma_2 - 1\}\quad \mbox{and} \quad
\rho_3 = \max \{2\sigma_1 + 4, 2\sigma_2 + 1, 2\sigma_3 - 1\}.
\end{align*}
We refer to  \cite{feng2001numerical, feng2004error} for their detailed proof.
\end{remark}

\subsection{Energy decay structure}
In this subsection, we prove the following energy decay property of the numerical solution, which comprise of the first theorem of this paper.  
\begin{theorem}{(energy decay)}\label{THM:Energy-decay}
The scheme \eqref{SAV-Euler-v1} is unconditionally energy stable in the sense that
\begin{align}\label{Energy-decay-un}
E(u^{n+1}, r^{n+1}) - E(u^n, r^n) \le - \frac{1}{\tau}\|u^{n+1} - u^n\|_{-1}^2 \le 0 \quad\mbox{for } n\geq 1.
\end{align}
\end{theorem}

\begin{proof}
Taking the inner product of the first equation in \eqref{SAV-Euler-v1} with $-\Delta^{-1}(u^{n+1} - u^n)$, and of the second equation with $u^{n+1} - u^n$, and multiplying the third equation in \eqref{SAV-Euler-v1} by $\frac{2}{\varepsilon} r^{n+1}$, we derive that
\begin{align} 
& \frac{1}{\tau} \|u^{n+1} - u^n\|_{-1}^2 
+ \frac{\varepsilon}{2} \Big(\|\nabla u^{n+1}\|^2 - \|\nabla u^n\|^2   + \|\nabla u^{n+1} - \nabla u^n\|^2 \Big)\\
& \qquad \qquad + \frac{1}{\varepsilon}r^{n+1} \big( g(u^n), u^{n+1} - u^n \big) = 0,  \notag \\
&\frac{1}{\varepsilon} \Big((r^{n+1})^2 - (r^n)^2 +  (r^{n+1} - r^n)^2 \Big)
= \frac{1}{\varepsilon}r^{n+1} \big( g(u^n), u^{n+1} - u^n \big). \label{energy-r-1}
\end{align}
Taking the summation of the above equations, we get
\begin{align}\label{Energy-decay-est1}
\frac{1}{\tau} \|u^{n+1} - u^n\|_{-1}^2 
+ & \frac{\varepsilon}{2} \Big(\|\nabla u^{n+1}\|^2 - \|\nabla u^n\|^2   + \|\nabla u^{n+1} - \nabla u^n\|^2 \Big)\\
+ & \frac{1}{\varepsilon} \Big((r^{n+1})^2 - (r^n)^2 +  (r^{n+1} - r^n)^2 \Big)
= 0. \notag
\end{align}
which gives \eqref{Energy-decay-un}.
\end{proof}

\begin{remark}
After summing up \eqref{Energy-decay-est1} from $n = 0$ to $N$, we get
\begin{align}\label {Energy-decay-sta}
\frac{\varepsilon}{2} \|\nabla u^{N+1}\|^2  +  \frac{1}{\varepsilon} (r^{N+1})^2  
& + \frac{1}{\tau} \sum_{n = 0}^N\|u^{n+1} - u^n\|_{-1}^2  
+ \frac{\varepsilon}{2} \sum_{n = 0}^N \|\nabla u^{n+1} - \nabla u^n\|^2 \\
& + \frac{1}{\varepsilon}\sum_{n = 0}^N (r^{n+1} - r^n)^2
=  \frac{\varepsilon}{2}  \|\nabla u^0\|^2 + \frac{1}{\varepsilon}  (r^0)^2 
\lesssim   \varepsilon^{-\rho_1}. \notag
\end{align}
It follows from $u^{n+1} - u^n \in L_0^2(\Omega)$ and \eqref{inq-l2tohm1} that
\begin{align}\label{sta-un-l2}
\sum_{n = 0}^N\|u^{n+1} - u^n\|^2  \le \, & \sum_{n = 0}^N\|u^{n+1} - u^n\|_{-1}\|\nabla u^{n+1} - \nabla u^n\| \\
\le \, & \Big(\sum_{n = 0}^N\|u^{n+1} - u^n\|_{-1}^2  \Big)^{\frac12}
\Big(\sum_{n = 0}^N\|\nabla u^{n+1} - \nabla u^n\|^2  \Big)^{\frac12} \notag\\
\lesssim \, &  \varepsilon^{-(\rho_1 + \frac12)}\tau^{\frac12}. \notag
\end{align}
\end{remark}

\subsection{Consistency}
To derive the error estimates of the equivalent transformation \eqref{sav-refor-scheme} of the semi-discrete SAV scheme \eqref{SAV-Euler-v1}, we reformulate the CH equation \eqref{pde-CH} as the truncated form
\begin{align}\label{pde-CH-truncated}
\Delta^{-1} \frac{u(t_{n+1}) - u(t_n)}{\tau} & = - \varepsilon \Delta u(t_{n+1}) + \frac{1}{\varepsilon} f(u(t_n)) + \mathcal{R}^{n+1},  
\end{align}
where the truncation error $\mathcal{R}^{n+1}$ is given by
\begin{align}\label{pde-CH-Rn}
\mathcal{R}^{n+1} 
:= \Big[ \Delta^{-1} \frac{u(t_{n+1}) - u(t_n)}{\tau} - \Delta^{-1} \partial_t u(t_{n+1}) \Big]
+ \frac{1}{\varepsilon}\Big[f(u(t_{n+1})) - f(u(t_n))\Big].  
\end{align}
\begin{lemma}{(consistency estimate)}\label{lem-consistency-Rn} 
Suppose that assumptions \ref{ass-f} and \ref{ass-u0} hold, then we have the following consistency estimate:
\begin{align}\label{consistency-Rn}
\tau \sum_{n = 0}^N \|\mathcal{R}^{n+1} \|_{H^{-1}}^2 \le C \varepsilon^{- \max\{\rho_2+2, \,\rho_3 \}} \tau^2.
\end{align}
\end{lemma}
\begin{proof}
For any $\varphi \in L_0^2(\Omega)$, there holds $\varphi_1 = - \Delta^{-1} \varphi \in L_0^2(\Omega) \cap H^1(\Omega)$ with $\partial_{\bm n} \varphi_1 = 0$ on $\partial \Omega$  and $\|\varphi\|_{-1} = \|\nabla \varphi_1\|$, which gives
\begin{align}
\|\varphi\|_{H^{-1}} & = \sup_{v \,\in H^1(\Omega)} \frac{(\varphi, v )}{\|v\|_{H^1}}
= \sup_{v \,\in H^1(\Omega)} \frac{\big(-\Delta (-\Delta^{-1})\varphi, v \big) }{\|v\|_{H^1}}
= \sup_{v \,\in H^1(\Omega)} \frac{\big(\nabla \varphi_1, \nabla v \big) }{\|v\|_{H^1}} \\
& \le \|\nabla \varphi_1\| = \|\varphi\|_{-1}. \notag
\end{align}
By performing standard calculations, it follows from $\Delta^{-1} u_{tt} \in L_0^2(\Omega)$ that 
\begin{align}
\Big \| \Delta^{-1} \frac{u(t_{n+1}) - u(t_n)}{\tau} - \Delta^{-1} \partial_t u(t_{n+1}) \Big\|_{H^{-1}}^2 
= \, & \Big \| \frac{1}{\tau}\int_{t_n}^{t_{n+1}} (s - t_n) \Delta^{-1} u_{tt} \, \d s  \Big\|_{H^{-1}}^2 \\
\lesssim \, & \tau \int_{t_n}^{t_{n+1}} \| \Delta^{-1} u_{tt}  \|_{H^{-1}}^2 \, \d s \notag\\
\le \, &  \tau \int_{t_n}^{t_{n+1}} \| \Delta^{-1} u_{tt}  \|_{-1}^2 \, \d s, \notag
\end{align}
and
\begin{align}
\Big \| \frac{1}{\varepsilon} f(u(t_{n+1})) -  \frac{1}{\varepsilon} f(u(t_n)) \Big\|_{H^{-1}}^2 
= \, & \varepsilon^{-2} \sup_{v \, \in H^1(\Omega)} \frac{\big(f(u(t_{n+1})) - f(u(t_n)), v \big)^2 }{\|v\|_{H^1}^2} \\
\le \, & \varepsilon^{-2} \sup_{v \,\in H^1(\Omega)} \frac{  \|f'(\xi^n)\|_{L^3}^2 \|u(t_{n+1}) - u(t_n)\|^2 \|v\|_{L^6}^2 }{\|v\|_{H^1}^2} \notag\\
\lesssim \, & \varepsilon^{-2}\|u(t_{n+1}) - u(t_n)\|^2 \notag\\
\lesssim \, & \varepsilon^{-2} \tau \int_{t_n}^{t_{n+1}} \|u_t\|^2\, \d s, \notag
\end{align}
where $\xi^n$ is between $u(t_n)$ and $u(t_{n+1})$.
Thus, we have
\begin{align}
\tau \sum_{n = 0}^N \|\mathcal{R}^{n+1} \|_{H^{-1}}^2 
\lesssim \, & \tau^2 \int_{0}^{T} \| \Delta^{-1} u_{tt}  \|_{-1}^2 \, \d s 
+ \varepsilon^{-2} \tau^2 \int_{0}^{T} \|u_t\|^2\, \d s\\
\lesssim \, &  \varepsilon^{- \rho_3}\tau^2 +  \varepsilon^{- (\rho_2 + 2)} \tau^2. \notag 
\end{align}
The proof is completed.
\end{proof}

\section{Error estimates}\label{sec-4}
In this section,  we will derive the error bound of the semi-discrete scheme \eqref{SAV-Euler-v1}, in which the focus is to obtain the polynomial type dependence of the error bound on $\varepsilon^{-1}$.
If we use the usual error estimate of the SAV numerical scheme \eqref{SAV-Euler-v1}, the error growth depends on $\varepsilon^{-1}$ exponentially. To avoid the exponential dependence on $\varepsilon^{-1}$ induced by using the Gronwall inequality, we need to use a spectral estimate of the linearized Cahn-Hilliard operator, which is given in \cite{alikakos1993spectrum, chen1994spectrum, feng2004error}.

\begin{lemma}{(spectral estimate)}\label{lem-spectral}
Suppose that Assumption \ref{ass-f} holds. Then there exist $0 < \varepsilon_0 << 1$ and a positive  constant $\lambda_0$ such that 
the principle eigenvalue of the linearized Cahn-Hilliard operator
\begin{align*}
\mathcal{L}_{CH} := \Delta (\varepsilon \Delta - \frac{1}{\varepsilon} f'(u) I)
\end{align*}
satisfies for all $t \in [0, T]$
\begin{align}\label{spectral-est}
\lambda_{CH} = \inf_{\substack{0 \neq v \in  H^1(\Omega) \\ \Delta w = v}}
\frac{\varepsilon\|\nabla v\|^2 + \frac{1}{\varepsilon}\big(f'(u(\cdot, t)v, v)\big)}{\|\nabla w\|^2} 
\ge -\lambda_0,
\end{align}   
for $\varepsilon \in (0, \varepsilon_0)$, where $I$ denotes the identity operator and $u$ is the solution of the Cahn-Hilliard problem \eqref{pde-CH}. 
\end{lemma}

We will now prove the following error estimates for the semi-discrete numerical scheme, which is the main result of this paper.
\begin{theorem}{(error estimate)}\label{THM:err-est}
We assume that assumptions \ref{ass-f} and \ref{ass-u0} hold and that 
\begin{equation}
\tau \le \tilde C \varepsilon^{\beta_0} \quad \mbox{with} \quad\beta_0 =
\frac{4\alpha_0 + 32 + 4d}{4 - d}, 
\end{equation}
then the discrete solution given by \eqref{SAV-Euler-v1} satisfies the following error estimate for $e^n = u(t_n) - u^n$:
\begin{align}\label{err-Hm1}
\max_{1 \le n \le m}\|e^n\|_{-1}^2  &+ \frac12 \sum_{n = 1}^m \|e^n - e^{n-1}\|_{-1}^2  + \frac12 \varepsilon^4\sum_{n = 1}^m  \tau \|\nabla e^n\|^2
+ \max_{1 \le n \le m} \tau \varepsilon^4 \|\nabla e^m\|^2 \\
&+  \frac12 \varepsilon^4 \sum_{n = 1}^m \tau\|\nabla  e^n - \nabla  e^{n-1}\|^2   \le \kappa_0\varepsilon^{- \alpha_0} \tau^2. \notag
\end{align}
where $\alpha_0 := \max\{ \rho_1 + 3, \, 2\rho_2 + 4, \, \rho_2 + 6, \, \rho_3 + 4\}$, and the constants $\tilde C $ and $\kappa_0$ are given in the proof.
\end{theorem}


\begin{proof}
We use the mathematical induction as follows. The proof is split into four steps. The first step gives the error estimate for the first step $t = t^1$. Steps  two and three use the spectral estimate \eqref{spectral-est} to avoid exponential blow-up in $\varepsilon^{-1}$ of the error constants. In the last step, an inductive argument is used to conclude the proof.

{\em Step 1: Estimation of $\|e^1\|_{-1}^2 + \tau \varepsilon \|\nabla e^1\|^2$}.
For $n = 0$ in \eqref{SAV-Euler}, we have
\begin{equation}
\left \{
\begin{aligned}
\Delta^{-1} \frac{u^1 - u^0}{\tau} & = - \varepsilon \Delta u^1 + \frac{1}{\varepsilon} r^1 g(u^0)  \\ 
r^1 & = r^0 + \frac12 \big(g(u^0), u^1 - u^0 \big) 
\end{aligned} 
\right .
\end{equation} 
After plugging the second equation into the first equation, we get
\begin{align}\label{sav-scheme-u1}
\Delta^{-1} \frac{u^1 - u^0}{\tau} & = - \varepsilon \Delta u^1 + \frac{1}{\varepsilon} f(u^0) 
+ \frac{1}{2\varepsilon} g(u^0)\big(g(u^0), u^1 - u^0 \big).
\end{align} 
Subtracting \eqref{sav-scheme-u1} from \eqref{pde-CH-truncated}, we get the corresponding error equation 
\begin{align}
\Delta^{-1} \frac{e^1}{\tau} & = - \varepsilon \Delta e^1 + \mathcal{R}^1 - \frac{1}{2\varepsilon} g(u^0)\big(g(u^0), u^1 - u^0 \big).
\end{align}
Taking the $L^2$ inner product of the above equation with $e^1 \in L_0^2(\Omega)$, we have 
\begin{align}
\|e^1\|_{-1}^2 + \tau \varepsilon \|\nabla e^1\|^2  =  - \tau (\mathcal{R}^1, e^1) + \frac{\tau}{2\varepsilon} \big(g(u^0), u^1 - u^0 \big)\big(g(u^0), e^1 \big).
\end{align}
Using Poincar\'e's inequality for $e^1 \in L_0^2(\Omega)$ and Lemma~\ref{lem-consistency-Rn}, we have
\begin{align}
\tau (\mathcal{R}^1, e^1)  \le \frac{\varepsilon \tau }{4}\|\nabla e^1\|^2 + C\varepsilon^{-1} \tau \|\mathcal{R}^1\|_{H^{-1}}^2 
\le \frac{\varepsilon \tau }{4} \|\nabla e^1\|^2 +   C\varepsilon^{- \max\{\rho_2+3, \,\rho_3 +1\}}\tau^2.
\end{align}
From \eqref{sta-un-l2}, we have
\begin{align}
\frac{\tau}{2\varepsilon} \big(g(u^0),  u^1 - u^0\big) \big(g(u^0), e^1\big)
\le\, & C \tau \varepsilon^{-1} \|g(u^0)\|^2\|u^1 - u^0\|\|e^1\|
\\
\le\, & C \tau \varepsilon^{-1}  \|u^1 - u^0\|
\|e^1\|_{-1}^{\frac12}\|\nabla e^1\|^{\frac12} \notag\\
\le \, & \frac12 \tau^{\frac12} \varepsilon^{\frac12}\|e^1\|_{-1} \|\nabla e^1\| 
+ C \tau^{\frac32} \varepsilon^{-\frac52} \|u^1 - u^0\|^2  \notag
\\
\le \, & \frac14 \|e^1\|_{-1}^2 +  \frac{\varepsilon \tau }{4} \|\nabla e^1\|^2 +  C\varepsilon^{-(\rho_1 + 3)}\tau^2 \notag
\end{align}
Thus, there exists a positive constant $\kappa_0'$ such that
\begin{align}\label{err-hm1-e1}
\|e^1\|_{-1}^2 + \tau \varepsilon \|\nabla e^1\|^2 \le \kappa_0' \varepsilon^{- \max\{\rho_1 + 3, \,\rho_2+3, \,\rho_3 +1\}} \tau^2.
\end{align}

We use the mathematical induction as follows.
For $m = 1$, \eqref{err-Hm1} holds from \eqref{err-hm1-e1}. 
We suppose  \eqref{err-Hm1} holds for $m = 1, 2, \dots, N$, and show that \eqref{err-Hm1} is also valid for $m = N+1$.

Subtracting \eqref{sav-refor-scheme} from \eqref{pde-CH-truncated}, by denoting $e^n = u(t_n) - u^n$, we get the error equation 
\begin{align}\label{err-eqn}
\Delta^{-1} \frac{e^{n+1} - e^n}{\tau} 
=\, & - \varepsilon \Delta e^{n+1} + \frac{1}{\varepsilon} \Big[f(u(t_n)) - f(u^n)\Big] \\
& + \frac{1}{2\varepsilon} g(u^n) \sum_{i = 1}^n D^2 A(\xi_i\,; u^i - u^{i - 1}) \notag
\\
& - \frac{1}{2\varepsilon} g(u^n) \big(g(u^n), u^{n+1} - u^n\big) + \mathcal{R}^{n+1}. \notag
\end{align}

{\em Step 2: Estimation  of $\displaystyle \|\nabla e^{N+1}\|^2 + \sum_{ n = 0}^N \|\nabla  e^{n+1} - \nabla  e^n\|^2 $.}
By taking the $L^2$ inner product of \eqref{err-eqn} with $e^{n+1} - e^n \in L_0^2(\Omega)$, we have
\begin{align}\label{err-h1-1}
\frac12\|\nabla e^{n+1}\|^2 - \,& \frac12\|\nabla e^n\|^2  + \frac12\|\nabla  e^{n+1} - \nabla  e^n\|^2 
+ \frac{\tau}{\varepsilon} \Big\|\frac{e^{n+1} - e^n}{\tau} \Big\|_{-1}^2 \\
=\,& - \frac{1}{\varepsilon^2} \Big(f(u(t_n)) - f(u^n), e^{n+1} - e^n\Big) \notag\\
& - \frac{1}{2\varepsilon^2} \big(g(u^n), e^{n+1} - e^n\big) \sum_{i = 1}^n D^2 A(\xi_i\,; u^i - u^{i - 1}) \notag
\\
& + \frac{1}{2\varepsilon^2} \big(g(u^n), e^{n+1} - e^n\big) (g(u^n), u^{n+1} - u^n) - \frac{1}{\varepsilon}\big(\mathcal{R}^{n+1}, e^{n+1} - e^n\big) \notag\\
=: \,& K_1 + K_2 + K_3 + K_4. \notag
\end{align}
It follows from \eqref{inq-l2tohm1} that
\begin{align}\label{err-h1-k1}
K_1 = \, & - \frac{1}{\varepsilon^2} \Big(f(u(t_n)) - f(u^n), e^{n+1} - e^n\Big) \\
\le\,& \varepsilon^{-2} \|f'(w^n)\|_{L^\infty} \|e^n\|\|e^{n+1} - e^n\| \notag\\
\le \,& C\varepsilon^{-2} \|\nabla e^n\|^{\frac12} \|e^n\|_{-1}^{\frac12}\|\nabla e^{n+1} - \nabla e^n\|^{\frac12} \|e^{n+1} - e^n\|_{-1}^{\frac12}\notag \\
\le \,& C\frac{1}{\gamma_0} \varepsilon^{-2}\tau^{\frac12} \|\nabla e^n\| \|e^n\|_{-1}
+ \gamma_0\varepsilon^{-2} \tau^{-\frac12} \|\nabla e^{n+1} - \nabla e^n\|\|e^{n+1} - e^n\|_{-1}\notag\\
\le \,& C\tau \|\nabla e^n\|^2 + C \varepsilon^{-4} \|e^n\|_{-1}^2
+  \frac12 \gamma_0\tau^{-1}\varepsilon^{-4} \|e^{n+1} - e^n \|_{-1}^2 
+ \frac12 \gamma_0 \|\nabla e^{n+1} - \nabla e^n\|^2 \notag
\end{align}
where $w_n = \theta u^n + (1 - \theta) u(t_n)$ with $\theta \in (0, 1)$,  and $\gamma_0$ is a  sufficiently small constant such that the two terms involving $\|e^{n+1} - e^n\|_{-1}^2$ and $\|\nabla e^{n+1} - \nabla e^n\|^2$ can be absorbed by the left hand side in the Step 4.

By using  $D^2 A(v; w) \lesssim \big(\|f'(v)\|_{L^\infty} + \|f(v)\|^2 \big) \|w\|^2$, the term $K_2$ can be bounded as
\begin{align}
K_2 = \,& - \frac{1}{2\varepsilon^2} \big(g(u^n), e^{n+1} - e^n\big) \sum_{i = 1}^n D^2 A(\xi_i\,; u^i - u^{i - 1}) \\
\lesssim\, &  \frac{1}{2\varepsilon^2} \big|\big(g(u^n), e^{n+1} - e^n\big)\big| \sum_{i = 1}^n \|u^i - u^{i - 1}\|^2 \notag\\
\lesssim \, &  \frac{1}{\varepsilon^2} \big|\big(g(u^n), e^{n+1} - e^n\big)\big| \sum_{i = 1}^n \|e^i - e^{i - 1}\|^2
+ \frac{1}{\varepsilon^2} \big|\big(g(u^n), e^{n+1} - e^n\big)\big| \sum_{i = 1}^n\| u(t_i) - u(t_{i - 1})\|^2 \notag\\
=:\,& K_{21} + K_{22}. \notag
\end{align}
From the assumptions of induction,  we have
\begin{align}\label{err-h1-est1}
\big(g(u^n), e^{n+1} - e^n\big) & =  \frac{1}{ \sqrt{\int_{\Omega} F(u^n) \, \d x + c_0}}
\big(f(u^n), e^{n+1} - e^n\big) \\
& \le  C\|f(u^n)\| \|e^{n+1} - e^n\| \notag\\
& \le  C\|\nabla e^{n+1} - \nabla e^n\|^{\frac12} \|e^{n+1} - e^n\|_{-1}^{\frac12},
\notag\\
\label{err-h1-est2}
\sum_{i = 1}^n  \|e^i - e^{i-1}\|^2
&  \lesssim \Big(\sum_{i = 1}^n \|e^i - e^{i-1}\|_{-1}^2 \Big)^{\frac12} 
\Big(\sum_{i = 1}^n \|\nabla e^i - \nabla e^{i-1}\|^2 \Big)^{\frac12} \\
& \lesssim  \kappa_0 \varepsilon^{-(\alpha_0 + 2)}\tau^{\frac32}, \notag\\
\label{err-h1-est3}
\sum_{i = 1}^n \|u(t_i) - u(t_{i - 1})\|^2
& \lesssim  \tau \sum_{i = 1}^n  \int_{t_{i-1}}^{t_i} \|u_t\|^2 \, \d s 
\lesssim   \tau \int_{0}^{t_n} \|u_t\|^2 \, \d s \\
& \lesssim  \varepsilon^{-\rho_2}\tau. \notag
\end{align}
Using the estimates above, 
we have
\begin{align}
K_{21} 
\le \,& C \varepsilon^{-2} \|\nabla e^{n+1} - \nabla e^n\|^{\frac12} \|e^{n+1} - e^n\|_{-1}^{\frac12}   \kappa_0 \varepsilon^{-(\alpha_0 + 2)}\tau^{\frac32} \\
\le \, & \gamma_0 \varepsilon^{-2} \tau^{-\frac12} \|\nabla e^{n+1} - \nabla e^n\|\|e^{n+1} - e^n\|_{-1}
+ C \frac{1}{\gamma_0}\kappa_0^2 \varepsilon^{- (2\alpha_0 + 6)}\tau^{\frac72}  \notag \\
\le \,&  \frac12 \gamma_0 \tau^{-1} \varepsilon^{-4} \|e^{n+1} - e^n \|_{-1}^2
+  \frac12 \gamma_0 \|\nabla e^{n+1} - \nabla e^n\|^2
+  C\kappa_0^2 \varepsilon^{- (2\alpha_0 + 6)}\tau^{\frac72} . \notag 
\end{align}
and
\begin{align}
K_{22} 
\le \, & C\varepsilon^{-2} \|e^{n+1} - e^n\|_{-1}^{\frac12} \|\nabla e^{n+1} - \nabla e^n\|^{\frac12}  \varepsilon^{-\rho_2}  \tau \\
\le \, &   \gamma_0 \varepsilon^{-2}\tau^{-\frac12} \|\nabla e^{n+1} - \nabla e^n\|\|e^{n+1} - e^n\|_{-1}
+ C \frac{1}{\gamma_0} \varepsilon^{- (2\rho_2 + 2)} \tau^{\frac52} \notag \\
\le \, &  \frac12  \gamma_0 \tau^{-1} \varepsilon^{-4} \|e^{n+1} - e^n \|_{-1}^2
+ \frac12  \gamma_0 \|\nabla e^{n+1} - \nabla e^n\|^2
+  C \varepsilon^{- (2\rho_2 + 2)} \tau^{\frac52} . \notag 
\end{align}
In addition, we divided the term $K_3$ into two parts as
\begin{align}
K_3 = \,& \frac{1}{2\varepsilon^2} \big(g(u^n), e^{n+1} - e^n\big) \big(g(u^n), u^{n+1} - u^n\big) \\
= \,&  -\frac{1}{2\varepsilon^2} \big(g(u^n), e^{n+1} - e^n\big) \big(g(u^n), e^{n+1} - e^n\big) \notag\\
& + \frac{1}{2\varepsilon^2} \big(g(u^n), e^{n+1} - e^n\big) \big(g(u^n), u(t_{n+1}) - u(t_n)\big) \notag \\
=:\, & K_{31} + K_{32}, \notag
\end{align}
where
\begin{align}
K_{31} = \, & - \frac{1}{2\varepsilon^2} \big(g(u^n), e^{n+1} - e^n\big) \big(g(u^n), e^{n+1} - e^n\big) \\
\le \, & C \varepsilon^{-2} \|g(u^n)\|^2 \|e^{n+1} - e^n\|_{-1} \|\nabla e^{n+1} - \nabla e^n\|  \notag \\
\le \, & C \frac{1}{\gamma_0}\varepsilon^{-4} \|e^{n+1} - e^n \|_{-1}^2
+ \frac12 \gamma_0\|\nabla e^{n+1} - \nabla e^n\|^2, \notag
\end{align}
and
\begin{align}
K_{32} =  \, & \frac{1}{2\varepsilon^2} \big(g(u^n), e^{n+1} - e^n\big) \big(g(u^n), u(t_{n+1}) - u(t_n)\big)\\ 
\le \, & C \varepsilon^{-2} \|e^{n+1} - e^n\|_{-1}^{\frac12}\|\nabla e^{n+1} - \nabla e^n\|^{\frac12}  \|u(t_{n+1}) - u(t_n)\| \notag \\
\le \, &  \gamma_0 \varepsilon^{-2}
\tau^{-\frac12} \|e^{n+1} - e^n\|_{-1}\|\nabla e^{n+1} - \nabla e^n\|
+  C \gamma_0^{-1}\varepsilon^{- 2}\tau^{\frac32}  \int_{t_n}^{t_{n+1}}\|u_t\|^2 \, \d s \notag \\
\le \, & \frac12 \gamma_0\tau^{-1} \varepsilon^{-4} \|e^{n+1} - e^n \|_{-1}^2
+ \frac12 \gamma_0 \|\nabla e^{n+1} - \nabla e^n\|^2
+ C\varepsilon^{- 2}\tau^{\frac32}  \int_{t_n}^{t_{n+1}}\|u_t\|^2 \, \d s , \notag 
\end{align}
By using Poincar\'e's inequality for $e^{n+1} -  e^n \in L_0^2(\Omega)$, the last term on the right hand side of \eqref{err-h1-1} can be bounded by
\begin{align}\label{err-h1-k4}
K_4 = \frac{1}{\varepsilon}\big(\mathcal{R}^{n+1}, e^{n+1} - e^n\big) 
\le   \frac12 \gamma_0 \|\nabla e^{n+1} - \nabla e^n\|^2 
+ C\frac{1}{\gamma_0} \varepsilon^{-2} \|\mathcal{R}^{n+1}\|_{H^{-1}}^2.
\end{align}
Combining these estimates \eqref{err-h1-k1}--\eqref{err-h1-k4} together with \eqref{err-h1-1}, and taking the summation for $n = 0$ to $N$, we have
\begin{align}\label{err-h1-result}
\|\nabla e^{N+1}\|^2  & +  (1 - 6 \gamma_0) \sum_{n = 0}^N \|\nabla  e^{n+1} - \nabla  e^n\|^2
\le  C\tau \sum_{n = 0}^N  \|\nabla e^n\|^2 + C \varepsilon^{-4}\sum_{n = 0}^N  \|e^n\|_{-1}^2 \\
& 
+  (C + 4 \gamma_0 \tau^{-1})\varepsilon^{-4}\sum_{n = 0}^N  \|e^{n+1} - e^n \|_{-1}^2 +  C\varepsilon^{-2} \sum_{n = 0}^N \|\mathcal{R}^{n+1}\|_{H^{-1}}^2
\notag\\
& +  C\kappa_0^2 \varepsilon^{- (2\alpha_0 + 6)}\tau^{\frac52} 
+  C  \varepsilon^{- (2\rho_2 + 2)} \tau^{\frac32} +  C \varepsilon^{-2}\tau^{\frac32}  \int_0^T\|u_t\|^2 \, \d s.
\notag
\end{align}

{\em  Step 3: Estimation of $\displaystyle \|e^{N+1}\|_{-1}^2 + \sum_{n = 0}^N \|e^n - e^{n-1}\|_{-1}^2 $.}
Taking the $L^2$ inner product of \eqref{err-eqn} with $e^{n+1} \in L_0^2(\Omega)$, we get
\begin{align}\label{err-hm1-1}
\frac{1}{2\tau} \Big(\|e^{n+1}\|_{-1}^2 & -  \|e^n\|_{-1}^2  + \|e^{n+1} -  e^n\|_{-1}^2 \Big)
+ \varepsilon \|\nabla e^{n+1}\|^2 
+ \frac{1}{\varepsilon} \Big(f(u(t_n)) - f(u^n), e^{n+1}\Big) 
\notag\\
& + \frac{1}{2\varepsilon} \big(g(u^n), e^{n+1} \big) \sum_{i = 1}^n D^2 A(\xi_i\,; u^i - u^{i - 1}) 
\\
& - \frac{1}{2\varepsilon} \big(g(u^n), e^{n+1}\big) \big(g(u^n), u^{n+1} - u^n \big) + \big(\mathcal{R}^{n+1}, e^{n+1} \big) = 0.\notag
\end{align}
We denote $w^n = \theta u^n + (1-\theta) u(t_n)$ with $\theta \in (0, 1)$, and using the Taylor expansion to get
\begin{align*}
\frac{1}{\varepsilon} \Big(f(u(t_n)) - f(u^n), e^{n+1}\Big) 
=\, & \frac{1}{\varepsilon} \Big(f'(u(t_n)) e^n - \frac12 f''(w^n)(e^n)^2, e^{n+1}\Big) \\
=\, & \frac{1}{\varepsilon} \Big(f'(u(t_n)) e^{n+1}, e^{n+1}\Big)
- \frac{1}{\varepsilon} \Big(f'(u(t_n)) (e^{n+1} - e^n), e^{n+1}\Big) \notag\\
& - \frac{1}{2\varepsilon} \Big(f''(w^n)(e^n)^2, e^{n+1}\Big).\notag 
\end{align*}
Then, \eqref{err-hm1-1} becomes
\begin{align}\label{err-hm1-2}
\frac{1}{2\tau}  \Big(\|e^{n+1}\|_{-1}^2  -\, &  \|e^n\|_{-1}^2  + \|e^{n+1} -  e^n\|_{-1}^2 \Big)
+ \varepsilon \|\nabla e^{n+1}\|^2 
+ \frac{1}{\varepsilon} \Big(f'(u(t_n)) e^{n+1}, e^{n+1}\Big) 
\\
= \, & \frac{1}{\varepsilon} \Big(f'(u(t_n)) (e^{n+1} - e^n), e^{n+1}\Big) 
+ \frac{1}{2\varepsilon} \Big(f''(w^n)(e^n)^2, e^{n+1}\Big) \notag\\
& - \frac{1}{2\varepsilon} \big(g(u^n), e^{n+1} \big) \sum_{i = 1}^n D^2 A(\xi_i\,; u^i - u^{i - 1}) 
\notag\\
& + \frac{1}{2\varepsilon} \big(g(u^n), e^{n+1}\big) \big(g(u^n), u^{n+1} - u^n \big) 
- \big(\mathcal{R}^{n+1}, e^{n+1} \big) \notag\\
=: \, & T_1 + T_2 + T_3 + T_4 + T_5. \notag
\end{align}
Using Lemma~\ref{lem-spectral}, there holds
\begin{align}\label{err-hm1-spectral}
\varepsilon \|\nabla e^{n+1}\|^2 
+ \frac{1}{\varepsilon} \Big(f'(u(t_n)) e^{n+1}, e^{n+1}\Big) \ge - \lambda_0\|e^{n+1}\|_{-1}^2.
\end{align}
If the entire $\varepsilon \|\nabla e^{n+1}\|^2 $ term is used to control the the term $\varepsilon^{-1}\big(f'(u(t_n)) e^{n+1}, e^{n+1}\big)$, we will not
be able to control the $\|\nabla e^{n+1}\|^2 $ terms in $T_j$, $j = 1, \dots, 5$. 
So we apply \eqref{err-hm1-spectral} with a scaling factor $(1 - \eta)$ close to but smaller than 1, to get
\begin{align}\label{err-hm1-spectral-p1}
- (1 - \eta)\frac{1}{\varepsilon} \Big(f'(u(t_n)) e^{n+1}, e^{n+1}\Big) 
\le (1 - \eta)\lambda_0 \|e^{n+1}\|_{-1}^2 +  (1 - \eta) \varepsilon \|\nabla e^{n+1}\|^2.
\end{align}
On the other hand, 
\begin{align}\label{err-hm1-spectral-p2}
- \frac{\eta}{\varepsilon} \Big(f'(u(t_n)) e^{n+1}, e^{n+1}\Big) 
\le \frac{C\eta}{\varepsilon} \|e^{n+1}\|^2 
\le \frac{C\eta}{\varepsilon^2 \eta_1} \|e^{n+1}\|_{-1}^2 + \frac{\eta \eta_1}{4}\|\nabla e^{n+1}\|^2.
\end{align}
The first term $T_1$ on the right-hand side of \eqref{err-hm1-2} can be bounded by 
\begin{align}\label{err-hm1-T1}
T_1 = \, & \frac{1}{\varepsilon} \Big(f'(u(t_n)) (e^{n+1} - e^n), e^{n+1}\Big)\\
\le \, & \varepsilon^{-1} \|f'(u(t_n))\|_{L^\infty} \|e^{n+1} - e^n\| \|e^{n+1}\| \notag\\
\le \, & C\varepsilon^{-1}  \|e^{n+1} - e^n\|_{-1}^{\frac12} \|\nabla e^{n+1} - \nabla e^n\|^{\frac12} \|e^{n+1}\|_{-1}^{\frac12} \|\nabla e^{n+1}\|^{\frac12} \notag\\
\le \, & \tau^{-\frac12} \|e^{n+1} - e^n\|_{-1} \|e^{n+1}\|_{-1} 
+ C \varepsilon^{- 2}\tau^{\frac12} \|\nabla e^{n+1} - \nabla e^n\|\|\nabla e^{n+1}\| \notag\\
\le \, & \frac12 \gamma_0 \tau^{-1} \|e^{n+1} - e^n\|_{-1}^2 + C\frac{1}{\gamma_0}\|e^{n+1}\|_{-1}^2
+ \frac{1}{16} \varepsilon^{\eta_2} \|\nabla e^{n+1}\|^2
+ T_1^*. \notag
\end{align}
where $ \displaystyle T_1^* : = C\varepsilon^{-(\eta_2 + 4)}\tau \|\nabla e^{n+1} - \nabla e^n\|^2$ and $\gamma_0$ is sufficiently small.
To control the last term $T_1^*$ on the right-hand side of \eqref{err-hm1-T1}, we assume that $\tau \le \Tilde C_1 \varepsilon^{\eta_2 + 8}$ to get
\begin{align*}
\tau\sum_{n = 0}^N T_1^* 
:= \, & C\varepsilon^{-(\eta_2 + 4)}\tau^2 \sum_{n = 0}^N  \|\nabla e^{n+1} - \nabla e^n\|^2 \notag\\
= \, & C\varepsilon^{-(\eta_2 + 8)}\tau \Big(\varepsilon^4\tau \sum_{n = 0}^N \|\nabla e^{n+1} - \nabla e^n\|^2 \Big)\notag \\
\le \, & \frac{1}{16} \varepsilon^4 \tau \sum_{n = 0}^N  \|\nabla e^{n+1} - \nabla e^n\|^2.
\end{align*}
Using the Sobolev interpolation inequality, we have for $v \in L_0^2(\Omega) \cap H^1(\Omega)$%
\begin{align}
\|v\|_{L^4} \le C\|\nabla v\|^{\frac{d}{4}}\|v\|^{1- \frac{d}{4}}
\le C\|\nabla v\|^{\frac{d}{4}} \|v\|_{-1}^{\frac12 - \frac{d}{8}}\|\nabla v\|^{\frac12 - \frac{d}{8} }= C\|\nabla v\|^{\frac12 + \frac{d}{8}} \|v\|_{-1}^{\frac12 - \frac{d}{8}},
\end{align}
which together with the assumptions of induction yields
\begin{align}\label{err-hm1-T2}
T_2 = \, & \frac{1}{2\varepsilon} \Big(f''(w^n)(e^n)^2, e^{n+1}\Big) \\
\le \, & C \varepsilon^{-1} \|e^n\|_{L^4}^2\|e^{n+1}\| \notag\\
\le \, & C \varepsilon^{-1} \|\nabla e^n\|^{1 + \frac{d}{4}} 
\|e^n\|_{-1}^{ 1 - \frac{d}{4}}
\|e^{n+1}\|_{-1}^{\frac12} \|\nabla e^{n+1}\|^{\frac12} 
\notag\\
\le \,&  C \kappa_0 \varepsilon^{- ( \alpha_0+ 3 + \frac{d}{2})} \tau^{\frac32 - \frac{d}{8}} \|e^{n+1}\|_{-1}^{\frac12} \|\nabla e^{n+1}\|^{\frac12} 
\notag\\
\le \, &  C \varepsilon^{\frac{\eta_2}{2}}\|\nabla e^{n+1}\| \|e^{n+1}\|_{-1}  
+   C\kappa_0^2 \varepsilon^{- (2\alpha_0 + 6 + d + \frac{\eta_2}{2})} \tau^{3 -\frac{d}{4}} 
\notag\\
\le \, & C\|e^{n+1}\|_{-1}^2 + \frac{1}{16} \varepsilon^{\eta_2}\|\nabla e^{n+1}\|^2
+   C\kappa_0^2 \varepsilon^{- (2\alpha_0 + 6 + d + \frac{\eta_2}{2})} \tau^{3 -\frac{d}{4}}. \notag
\end{align}
It follows from \eqref{err-h1-est2}-\eqref{err-h1-est3} that
\begin{align}\label{err-hm1-T3}
T_3 = \, & - \frac{1}{2\varepsilon} \big(g(u^n), e^{n+1} \big) \sum_{i = 1}^n D^2 A(\xi_i\,; u^i - u^{i - 1}) \\
\lesssim \, & \frac{1}{2\varepsilon}\big| \big(g(u^n), e^{n+1} \big)\big| \sum_{i = 1}^n \|u^i - u^{i - 1}\|^2 \notag\\
\le \, & \frac{1}{\varepsilon}\big| \big(g(u^n), e^{n+1} \big)\big| \sum_{i = 1}^n \|e^i - e^{i - 1}\|^2
+ \frac{1}{\varepsilon}\big| \big(g(u^n), e^{n+1} \big)\big| \sum_{i = 1}^n \| u(t_i) - u(t_{i - 1}) \|^2 \notag\\
\le \, & C\varepsilon^{-1} \|g(u^n)\|  \|e^{n+1}\| \Big(\kappa_0\varepsilon^{-(\alpha_0 + 2)} \tau^{\frac32}  + \varepsilon^{-\rho_2} \tau \Big)\notag
\\
\le \, & C\varepsilon^{-1} \|e^{n+1}\|_{-1}^{\frac12}  \|\nabla e^{n+1}\|^{\frac12}  \Big(\kappa_0\varepsilon^{-(\alpha_0 + 2)} \tau^{\frac32}  +  \varepsilon^{-\rho_2}\tau \Big) \notag
\\
\le \, & \varepsilon^{\frac{\eta_2}{2}}\|e^{n+1}\|_{-1}\|\nabla e^{n+1}\|
+  C \kappa_0^2 \varepsilon^{-(2\alpha_0 + 6 + \frac{\eta_2}{2})} \tau^3
+ C \varepsilon^{-(2\rho_2 + 2 + \frac{\eta_2}{2}) }\tau^2 \notag
\\
\le \, & C\|e^{n+1}\|_{-1}^2 + \frac{1}{16} \varepsilon^{\eta_2}\|\nabla e^{n+1}\|^2
+  C \kappa_0^2 \varepsilon^{-(2\alpha_0 + 6 + \frac{\eta_2}{2}) }\tau^3
+ C \varepsilon^{-(2\rho_2 + 2 + \frac{\eta_2}{2}) }\tau^2 . \notag
\end{align}
In order to estimate the term $T_4$, we divided it into two parts as
\begin{align}\label{err-hm1-T4}
T_4 = \, & \frac{1}{2\varepsilon} \big(g(u^n), e^{n+1}\big) \big(g(u^n), u^{n+1} - u^n \big) \\
= \, & - \frac{1}{2\varepsilon} \big(g(u^n), e^{n+1}\big) \big(g(u^n), e^{n+1} - e^n \big) \notag\\
& + \frac{1}{2\varepsilon} \big(g(u^n), e^{n+1}\big) \big(g(u^n), u(t_{n+1}) - u(t_n) \big) \notag\\
=:\, & T_{41} + T_{42}. \notag 
\end{align}
Similarly to the estimate of $T_1$, the term $T_{41}$ can be bounded by
\begin{align}\label{err-hm1-T41}
T_{41} = \, & - \frac{1}{2\varepsilon} \big(g(u^n), e^{n+1}\big) \big(g(u^n), e^{n+1} - e^n \big) \\
\le \, & C \varepsilon^{-1} \|g(u^n)\|^2 \|e^{n+1}\| \|e^{n+1} - e^n\|
\notag\\
\le \, & C\varepsilon^{-1}  \|e^{n+1} - e^n\|_{-1}^{\frac12} \|\nabla e^{n+1} - \nabla e^n\|^{\frac12} \|e^{n+1}\|_{-1}^{\frac12} \|\nabla e^{n+1}\|^{\frac12} \notag\\
\le \, &  \tau^{-\frac12} \|e^{n+1} - e^n\|_{-1}\|e^{n+1}\|_{-1} 
+ C \varepsilon^{-2}\tau^{\frac12} \|\nabla e^{n+1} - \nabla e^n\|\|\nabla e^{n+1}\| \notag\\
\le \, &  \frac12 \gamma_0\tau^{-1} \|e^{n+1} - e^n\|_{-1}^2 
+ C \frac{1}{\gamma_0} \|e^{n+1}\|_{-1}^2 + \frac{1}{16} \varepsilon^{\eta_2} \|\nabla e^{n+1}\|^2
+ T_1^*, \notag
\end{align}
which together with the condition $\tau \le \Tilde C_1 \varepsilon^{\eta_2 + 8}$ yields
\begin{align*}
\tau\sum_{n = 0}^N T_1^* 
=  C\varepsilon^{-(\eta_2 + 4)}\tau^2 \sum_{n = 0}^N  \|\nabla e^{n+1} - \nabla e^n\|^2
\le  \frac{1}{16} \varepsilon^4 \tau \sum_{n = 0}^N  \|\nabla e^{n+1} - \nabla e^n\|^2.
\end{align*}
In addition, we obtain
\begin{align}\label{err-hm1-T42}
T_{42} = \, & \frac{1}{2\varepsilon} \big(g(u^n), e^{n+1}\big) \big(g(u^n), u(t_{n+1}) - u(t_n) \big) \\
\le \, & C \varepsilon^{-1} \|\nabla e^{n+1}\|^{\frac12} \|e^{n+1}\|_{-1}^{\frac12} \|u(t_{n+1}) - u(t_n)\| \notag \\
\le \, & \varepsilon^{\frac{\eta_2}{2}}  \|\nabla e^{n+1} \|\|e^{n+1} \|_{-1}
+ C \varepsilon^{- (2 + \frac{\eta_2}{2})}\tau \int_{t_n}^{t_{n+1}}\|u_t\|^2 \, \d s \notag \\
\le \, & C\|e^{n+1}\|_{-1}^2 
+ \frac{1}{16}\varepsilon^{\eta_2}  \|\nabla e^{n+1}\|^2
+  C \varepsilon^{- (2 + \frac{\eta_2}{2})}\tau \int_{t_n}^{t_{n+1}}\|u_t\|^2 \, \d s. \notag  \notag 
\end{align}
For $T_5$, using Cauchy-Schwartz inequality and Poincar\'e's inequality for $e^{n+1} \in L_0^2(\Omega)$, we have
\begin{align}\label{err-hm1-T5}
T_5 = - \big(\mathcal{R}^{n+1}, e^{n+1} \big) 
\le \|\mathcal{R}^{n+1}\|_{H^{-1}}\|e^{n+1}\|_{H^1} 
\le 8\varepsilon^{-\eta_2} \|\mathcal{R}^{n+1}\|_{H^{-1}}^2 + \frac{1}{16} \varepsilon^{\eta_2} \|\nabla e^{n+1}\|^2.  
\end{align}
Combining the estimates \eqref{err-hm1-T1}--\eqref{err-hm1-T5} together with \eqref{err-hm1-2}, taking the summation for $n = 0$ to $N$, and assuming that $\tau \le \Tilde C_1 \varepsilon^{\eta_2 + 8}$, we have
\begin{align}\label{err-hm1-sum1}
\frac12 \|e^{N+1}\|_{-1}^2 &  + \frac12 ( 1 - 2\gamma_0) \sum_{n = 0}^N \|e^{n+1} -  e^n\|_{-1}^2
+ \varepsilon \tau \sum_{n = 0}^N \|\nabla e^{n+1}\|^2 \\
\le \, & \Big[C + (1-\eta) \lambda_0 + \frac{C\eta}{\varepsilon^2 \eta_1}\Big] \tau \sum_{n=0}^N \|e^{n+1}\|_{-1}^2
+ 8\varepsilon^{-\eta_2} \tau \sum_{n = 0}^{N} \|\mathcal{R}^{n+1}\|_{H^{-1}}^2 \notag\\
& + \Big[ (1-\eta)\varepsilon + \frac{\eta \eta_1}{4} + \frac{\varepsilon^{\eta_2}}{2} \Big] \tau \sum_{n=0}^N \|\nabla e^{n+1}\|^2 
+ \frac18 \varepsilon^4 \tau \sum_{n = 0}^N\|\nabla e^{n+1} - \nabla e^n\|^2 \notag\\
& + C\kappa_0^2 \varepsilon^{- (2\alpha_0 + 6 + d + \frac{\eta_2}{2})} \tau^{3 -\frac{d}{4}} 
+  C \kappa_0^2 \varepsilon^{-(2\alpha_0 + 6 + \frac{\eta_2}{2}) }\tau^3 \notag\\
& +  C \varepsilon^{-(2\rho_2 + 2 + \frac{\eta_2}{2}) }\tau^2 
+ C \varepsilon^{- (2 + \frac{\eta_2}{2})} \tau^2  \int_{0}^{T}\|u_t\|^2 \, \d s. \notag
\end{align}
By taking $\eta = \varepsilon^3$, $\eta_1 = \varepsilon$ and $\eta_2 = 4$, we have
\begin{align}
(1-\eta) \varepsilon + \frac{\eta \eta_1}{4} + \frac{\varepsilon^{\eta_2}}{2} = \varepsilon - \frac14 \varepsilon^4, 
\end{align}
which together with \eqref{sta-u-d} gives
\begin{align}\label{err-hm1-sum2}
\|e^{N+1}\|_{-1}^2 &  +  (1 - 2\gamma_0) \sum_{n = 0}^N \|e^{n+1} -  e^n\|_{-1}^2
+ \frac12 \varepsilon^4 \tau \sum_{n = 0}^N \|\nabla e^{n+1}\|^2 \\
\le \, & C\tau \sum_{n=0}^N \|e^{n+1}\|_{-1}^2 
+  \frac14 \tau \varepsilon^4 \sum_{n = 0}^N  \|\nabla e^{n+1} - \nabla e^n\|^2
+ C\varepsilon^{-4} \tau \sum_{n = 0}^{N} \|\mathcal{R}^{n+1}\|_{H^{-1}}^2 
\notag\\
& +   C\kappa_0^2 \varepsilon^{- (2\alpha_0 + 8 + d )} \tau^{3 -\frac{d}{4}} 
+ C \kappa_0^2 \varepsilon^{-(2\alpha_0 + 8) }\tau^3 \notag\\
& +  C\varepsilon^{-(2\rho_2 + 4) }\tau^2 
+ C\varepsilon^{-(\rho_2 + 4) }\tau^2 .
\notag
\end{align}

{\em Step 4: Completion of the proof.}
We now conclude the proof by the following induction argument which is based on the results from Step 1 to Step 3.
By multiplying $\tau \varepsilon^4$ on both sides of \eqref{err-h1-result}, combining the estimate \eqref{err-hm1-sum2}, and together with Lemma~\ref{lem-consistency-Rn}, we obtain
\begin{align}\label{err-hm1-plus-h1}
\|e^{N+1}\|_{-1}^2 & + \tau\varepsilon^4\|\nabla e^{N+1}\|^2  
+ (1 -  C \tau -  6\gamma_0)\sum_{n = 0}^N \|e^{n+1} -  e^n\|_{-1}^2 
\\
& +  (\frac34 - 6\gamma_0)\tau \varepsilon^4 \sum_{n = 0}^N \|\nabla  e^{n+1} - \nabla  e^n\|^2 \notag 
+ \frac12 \varepsilon^4 \tau \sum_{n = 0}^N \|\nabla e^{n+1}\|^2 \\
\le \, & C_0\tau \sum_{n=0}^N \Big(\|e^n\|_{-1}^2 + \tau\varepsilon^4 \|\nabla e^n\|^2 \Big)
+   C_1\kappa_0^2 \varepsilon^{- (2\alpha_0 + 8 + d )} \tau^{3 -\frac{d}{4}} 
\\
& +  C_2 \kappa_0^2 \varepsilon^{-(2\alpha_0 + 8) }\tau^3
+  C_3\kappa_0^2 \varepsilon^{- (2\alpha_0 + 2)}\tau^{\frac72} 
+ C_4\varepsilon^{-\max\{ \rho_1 + 3, \, 2\rho_2 + 4, \, \rho_2 + 6, \, \rho_3 + 4\} }\tau^2. 
\notag 
\end{align}
in which the term  $\kappa_0' \varepsilon^{- \max\{\rho_1 + 3, \,\rho_2+3, \,\rho_3 +1\}} \tau^2$ is absorbed in $ C_4\varepsilon^{-\max\{  \rho_1 + 3, \, 2\rho_2 + 4, \, \rho_2 + 6, \, \rho_3 + 4\} }\tau^2$.

Suppose that for sufficiently small constant $\gamma_0$ satisfying $\frac34 - 6\gamma_0 \ge \frac12$ and sufficiently small $\tau$ satisfying
\begin{align}
\tau \le\Big( \frac{C_4}{C_1} \kappa_0^{-2} \Big)^{\frac{4}{4-d}} \varepsilon^{\frac{4\alpha_0 + 32 + 4d}{4-d}},
\quad \tau \le \Big( \frac{C_4}{C_2} \kappa_0^{-2} \Big) \varepsilon^{\alpha_0 + 8}, 
\quad \tau \le \Big( \frac{C_4}{C_3} \kappa_0^{-2}\Big)^{\frac23} \varepsilon^{\frac{ (2\alpha_0 + 4)}{3}},
\end{align}
then, by denoting $\alpha_0 := \max\{ \rho_1 + 3, \, 2\rho_2 + 4, \, \rho_2 + 6, \, \rho_3 + 4\}$, 
we derive
\begin{align}
\|e^{N+1}\|_{-1}^2 & + \tau\varepsilon^4\|\nabla e^{N+1}\|^2  
+ \frac12\sum_{n = 0}^N \|e^{n+1} -  e^n\|_{-1}^2 
+ \frac12 \varepsilon^4  \sum_{n = 0}^N \tau\|\nabla e^{n+1}\|^2  \\
& +  \frac12\tau\varepsilon^4 \sum_{n = 0}^N \|\nabla  e^{n+1} - \nabla  e^n\|^2  \le \, C_0\tau \sum_{n=0}^N \Big(\|e^n\|_{-1}^2 + \tau\varepsilon^4 \|\nabla e^n\|^2 \Big) 
+ 4 C_4\varepsilon^{- \alpha_0 }\tau^2. \notag
\end{align}
We denote $\kappa_0 := 4C_4 e^{(C_0 T)}$ and use the Gronwall's inequality to get
\begin{align}\label{err-hm1-sum3}
\|e^{N+1}\|_{-1}^2 & + \frac12\sum_{n = 0}^N \|e^{n+1} -  e^n\|_{-1}^2 + \frac12 \varepsilon^4 \sum_{n = 0}^N \tau \|\nabla e^{n+1}\|^2 
+ \tau\varepsilon^4\|\nabla e^{N+1}\|^2  \\
& +  \frac12\tau\varepsilon^4 \sum_{n = 0}^m \|\nabla  e^{n+1} - \nabla  e^n\|^2 
\le \,   4C_4 e^{(C_0 T)} \varepsilon^{-\alpha_0} \tau^2 
= \kappa_0  \varepsilon^{-\alpha_0} \tau^2. \notag
\end{align}
The induction is completed.

In the above proof, we have used these conditions:
\begin{align}
\tau \le \Tilde C_1 \varepsilon^{12},
\quad  \tau \le \Tilde C_2 \varepsilon^{\frac{4\alpha_0 + 32 + 4d}{4-d}},
\quad \tau \le \Tilde C_3\varepsilon^{\alpha_0 + 8}, 
\quad \tau \le \Tilde C_4 \varepsilon^{\frac{ (2\alpha_0 + 4)}{3}},
\end{align}
where 
$\Tilde C_2 = \Big( \frac{C_4}{C_1} \kappa_0^{-2} \Big)^{\frac{4}{4-d}}$, 
$\Tilde C_3 = \Big( \frac{C_4}{C_2} \kappa_0^{-2} \Big)$, and 
$\Tilde C_4 = \Big( \frac{C_4}{C_3} \kappa_0^{-2}\Big)^{\frac23}$. 
By denoting 
\begin{equation}
\beta_0 = \frac{4\alpha_0 + 32 + 4d}{4-d}
\qquad \mbox{and} \qquad
\Tilde C  = \min \{\Tilde C_1, \Tilde C_2, \Tilde C_3, \Tilde C_4  \}, 
\end{equation}
we specify the final condition on $\tau$, that is, $\tau \le \Tilde C  \varepsilon^{
\beta_0}$.

\end{proof}

\section{Numerical experiments}\label{sec-5}
In this section, we present a two-dimensional numerical test to validate the theoretical results on
the energy decay properties proved in Theorem \ref{THM:Energy-decay}, as well as the convergence rates of the proposed method given in Theorem \ref{THM:err-est}. All the computations are 
performed using the software package NGSolve (\url{https://ngsolve.org}).

We solve the Cahn-Hilliard equation \eqref{pde-CH} on the two-dimensional square $\Omega = [0, 1] \times [0, 1]$ under Neumann boundary conditions by using the proposed scheme \eqref{SAV-Euler-v1} with 
the following initial condition
\begin{align}
u_0(x, y) = &\tanh \big(((x - 0.65)^2 + (y - 0.5)^2 - 0.1^2)/\varepsilon \big)\\
&\times \tanh \big(((x - 0.35)^2 + (y - 0.5)^2 - 0.125^2)/\varepsilon \big), \notag
\end{align}
where $\tanh(x) := (e^x - e^{-x})/(e^x + e^{-x})$. This type of initial condition is also adopted in \cite{feng2016analysis, feng2008posteriori}, where the set of the zero-level of the initial function $u_0(x, y)$ encloses two circles of radius $0.1$ and $0.125$, respectively. 

To obtain a $C^4$ potential function $F(v)$ that satisfies the assumption~\ref{ass-f}, we modify the common double-well potential $F(v) = \frac14 (v^2 - 1)^2$ by setting $M = 2$ in \eqref{truncate-F} to get a cut-off function $\hat F(v) \in C^4(\R) $. Correspondingly, the ninth-order polynomials $\Phi_+(v)$ and $\Phi_-(v)$ in \eqref{truncate-F} are determined with the following conditions
\begin{equation} 
\left \{
\begin{aligned}
&\Phi_+^{(i)}(M) = F^{(i)}(M)  \quad \mbox{and} \quad  \Phi_-^{(i)}(-M) = F^{(i)}(-M) \qquad \mbox{for} \,\,\, i = 0, 1, 2, 3, 4,\\
&\Phi_+(2M) = \Phi_-(-2M) = \frac14  ((2M)^2 - 1)^2, \\
&\Phi_+^{(1)}(2M) = \Phi_-^{(1)}(-2M) = (2M)^3 - 2M,  \\
&\Phi_+^{(i)}(2M) = \Phi_-^{(i)}(-2M) = 0   \qquad \mbox{for} \,\,\, i = 2, 3, 4.
\end{aligned}
\right .
\end{equation} 
Note that the truncation point $M = 2$ used here are for convenience only. 
For simplicity, we still denote the modified function $\hat F(u)$ by $F(u)$.

The spatial discretization is done by using the Galerkin finite element method. Let $S_h$ denotes the $P_s$ conforming finite element space defined by
\begin{align*}
S_h := \{ v_h \in C(\bar \Omega); \, \, v_h|_{K} \in P_s(K), \, \,\,  \forall\, K \in \mathcal{T}_h\},
\end{align*}
where $\mathcal{T}_h$ is a quasi-uniform triangulation of $\Omega$.
We introduce space notation $S_h^{\circ} := \{ v_h \in S_h ; \, \, (v_h, 1) = 0 \}$, and define the discrete inverse Laplace operator $- \Delta_h^{-1} : L_0^2(\Omega) \to S_h^{\circ} $ such that
\begin{align}
\big(\nabla (- \Delta_h^{-1} ) v, \nabla \eta_h \big) = (v, \eta_h) \qquad
\forall \, \, \, \eta_h \in S_h.
\end{align}
Since the exact solution of the considered problem is not known, we compute the orders of convergence by the formula
$$
\mbox{order of convergence} = 
\log\Bigg(\frac{\|u_N^{(\tau)}-u_N^{(\tau/2)}\|_{h, -1}}{\|u_N^{(\tau/2)}-u_N^{(\tau/4)}\|_{h, -1}}\Bigg)/\log(2) 
$$
based on the finest three meshes, where $u_N^{(\tau)}$ denotes the numerical solution at $t_N=T$ computed by using a stepsize $\tau$, and 
$\|v\|_{h, -1} := \sqrt{(v, - \Delta_h^{-1} v)}$ for $v \in L_0^2(\Omega)$.

The time discretization errors in $\|\cdot\|_{h, -1}$-norm  are presented in Figure \ref{fig_SAV_energy} (left) for four different $\varepsilon = 0.08, 0.06, 0.05, 0.04$ at $T = 0.005$, where we have used finite elements of degree $s = 3$ with a sufficiently spatial mesh $h = 1/64$ so that the error from spatial discretization is negligibly small in observing the temporal convergence rates. 
From  Figure \ref{fig_SAV_energy} (left),  we see that the error of time discretization  is $O(\tau)$, which is consistent with the theoretical results proved in Theorem \ref{THM:err-est}.  
In addition, Figure \ref{fig_SAV_energy} (right) shows the evolution in time of the discrete SAV energy   for four different $\varepsilon$, which should be
decreasing according to Theorem \ref{THM:Energy-decay}. This graph clearly confirms this decay property. 
Therefore, the numerical experiments are in accordance with our theoretical results. 
\begin{figure}[htp]
\centerline{
\includegraphics[width=3.2in]{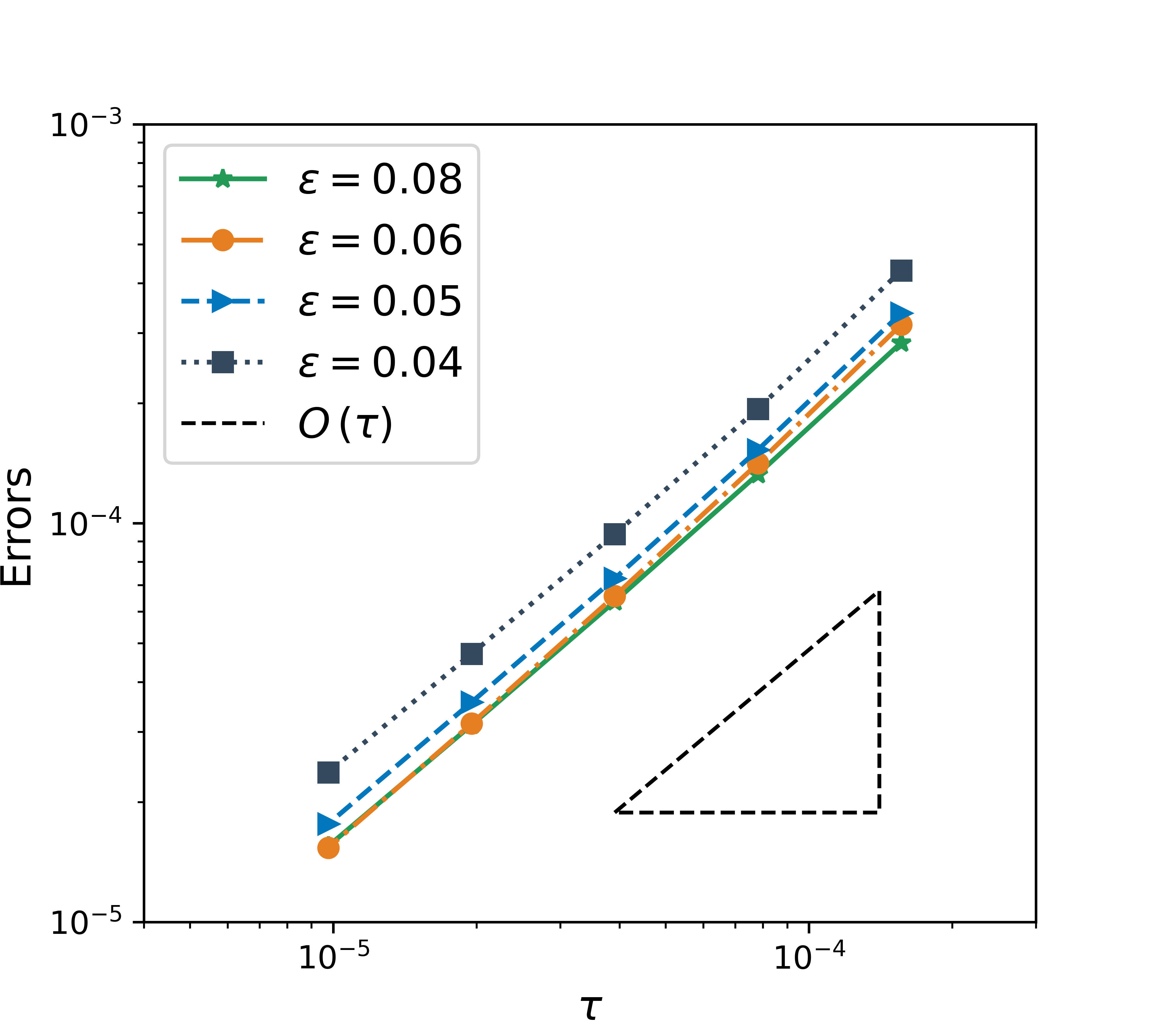}\hspace{-8pt}
\includegraphics[width=3.2in]{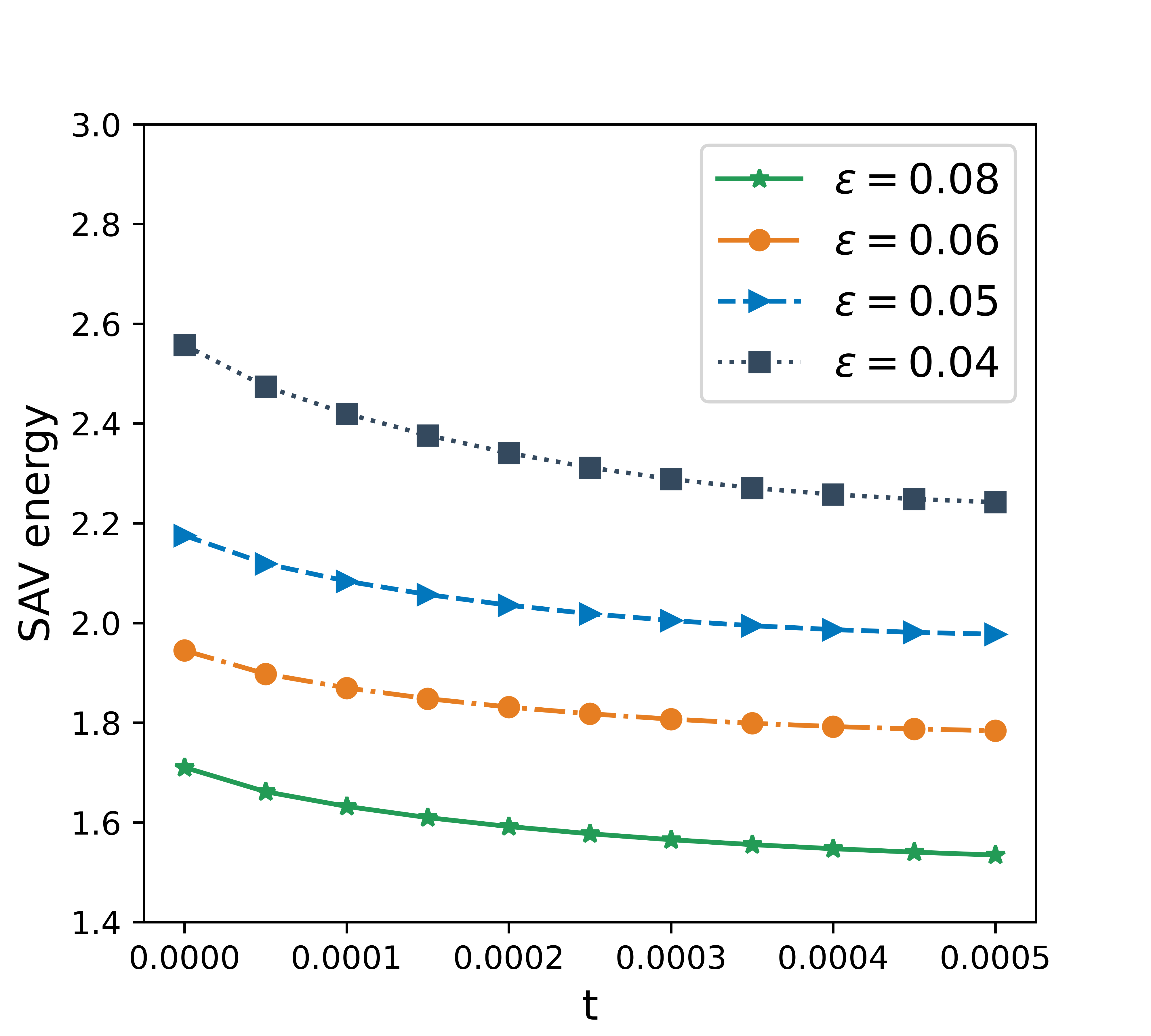}
}
\vspace{-8pt}
\caption{(left) Time discretization errors; \,\, (right) Evolution of the SAV energy.}
\label{fig_SAV_energy}
\end{figure}

Figure \ref{fig_zero_level} shows snapshots of the numerical interface for four different $\varepsilon=0.08, 0.06, 0.05, 0.04$ at six fixed time points. They clearly indicate that at each time point, as $\varepsilon$ tends to zero, the numerical interface converges to the sharp interface of the Hele-haw flow, which is consistent with the phenomenon stated in \cite{feng2016analysis, feng2008posteriori}.
It also shows that for larger $\varepsilon$, the numerical interface evolves faster in time.
%

%

\begin{figure}[htp]
\centerline{
\includegraphics[width=3.2in]{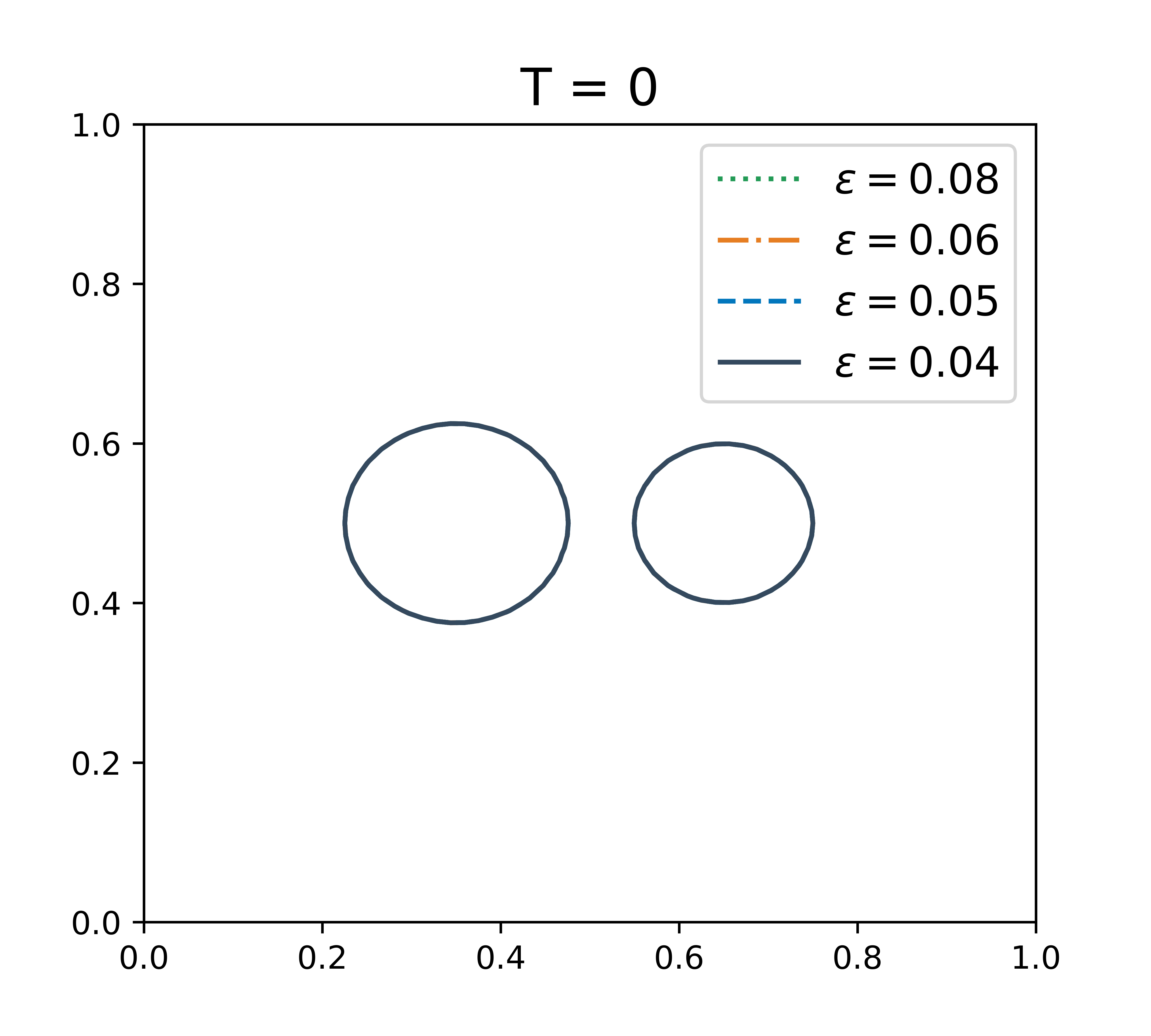}\hspace{-8pt}
\includegraphics[width=3.2in]{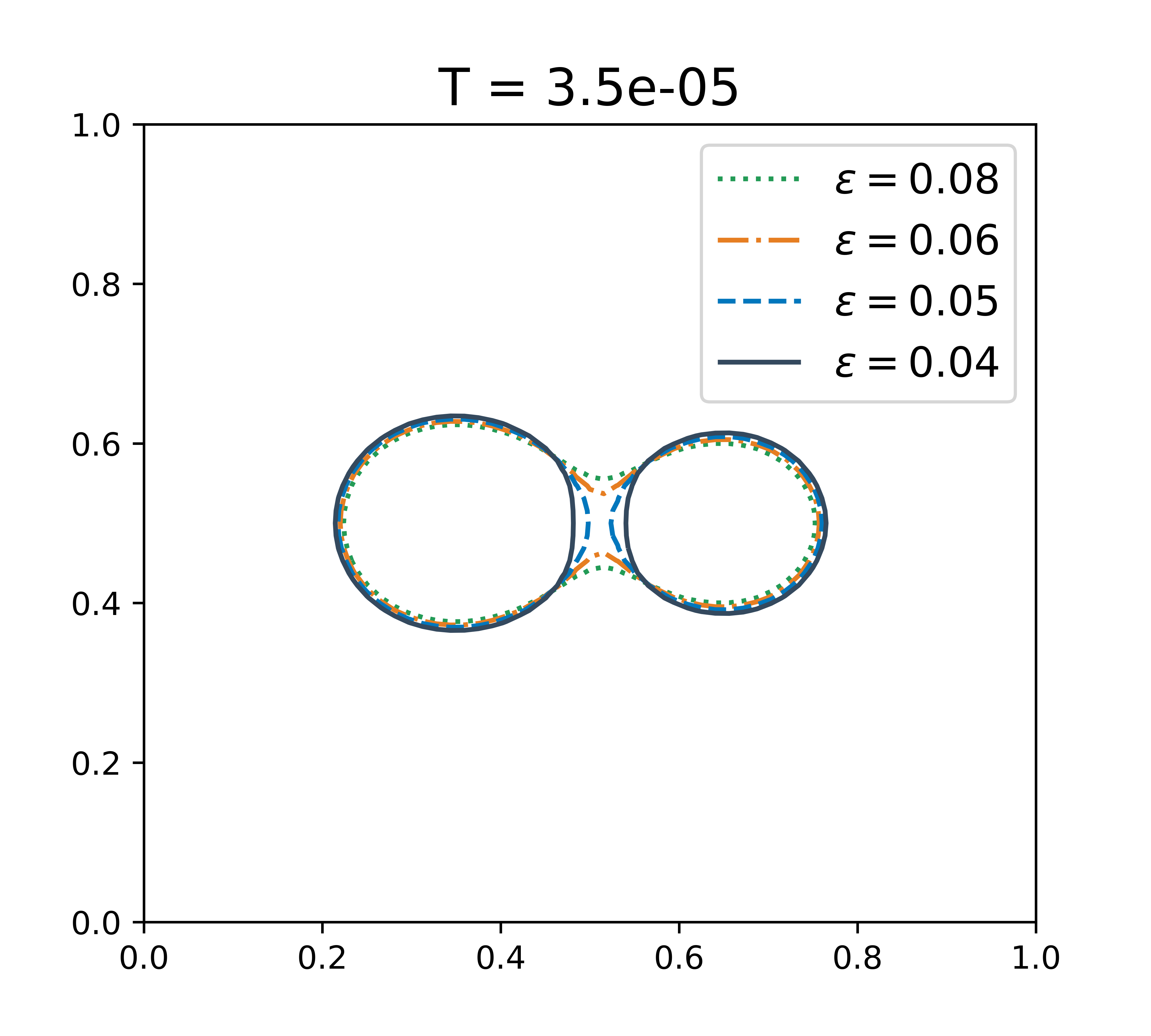}
}
\vspace{-10pt}
\centerline{
\includegraphics[width=3.2in]{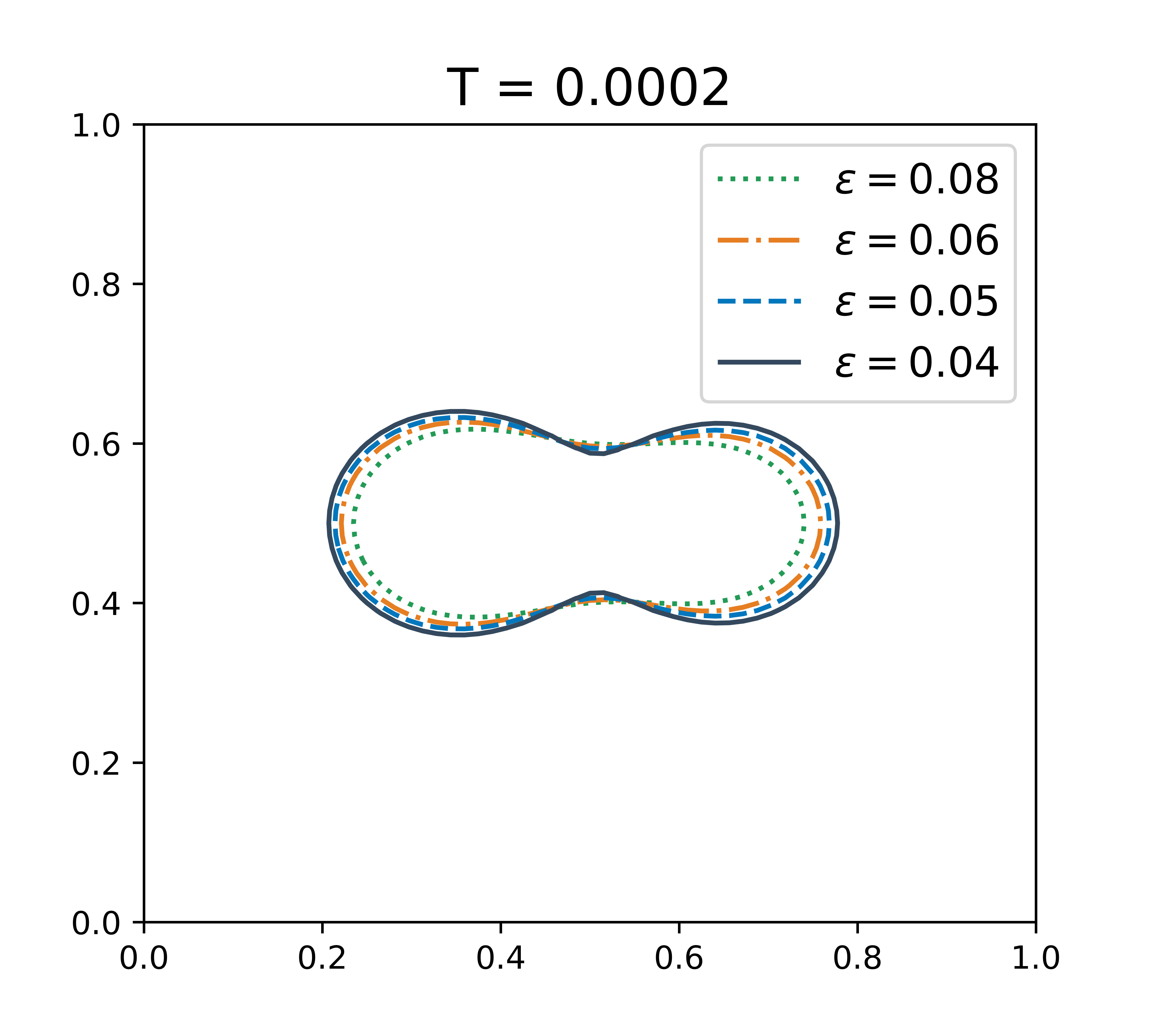}\hspace{-8pt}
\includegraphics[width=3.2in]{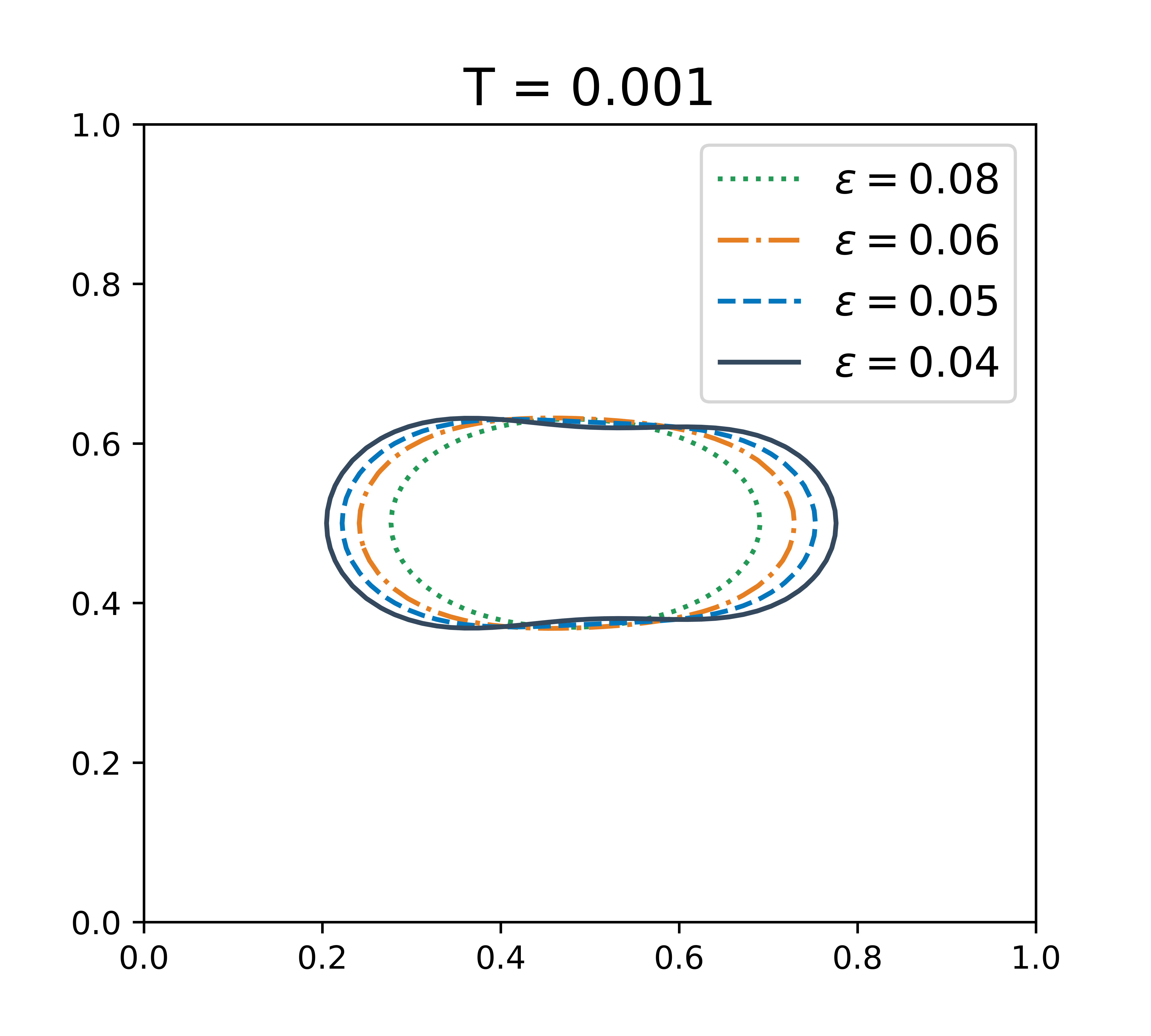}
}
\vspace{-10pt}
\centerline{
\includegraphics[width=3.2in]{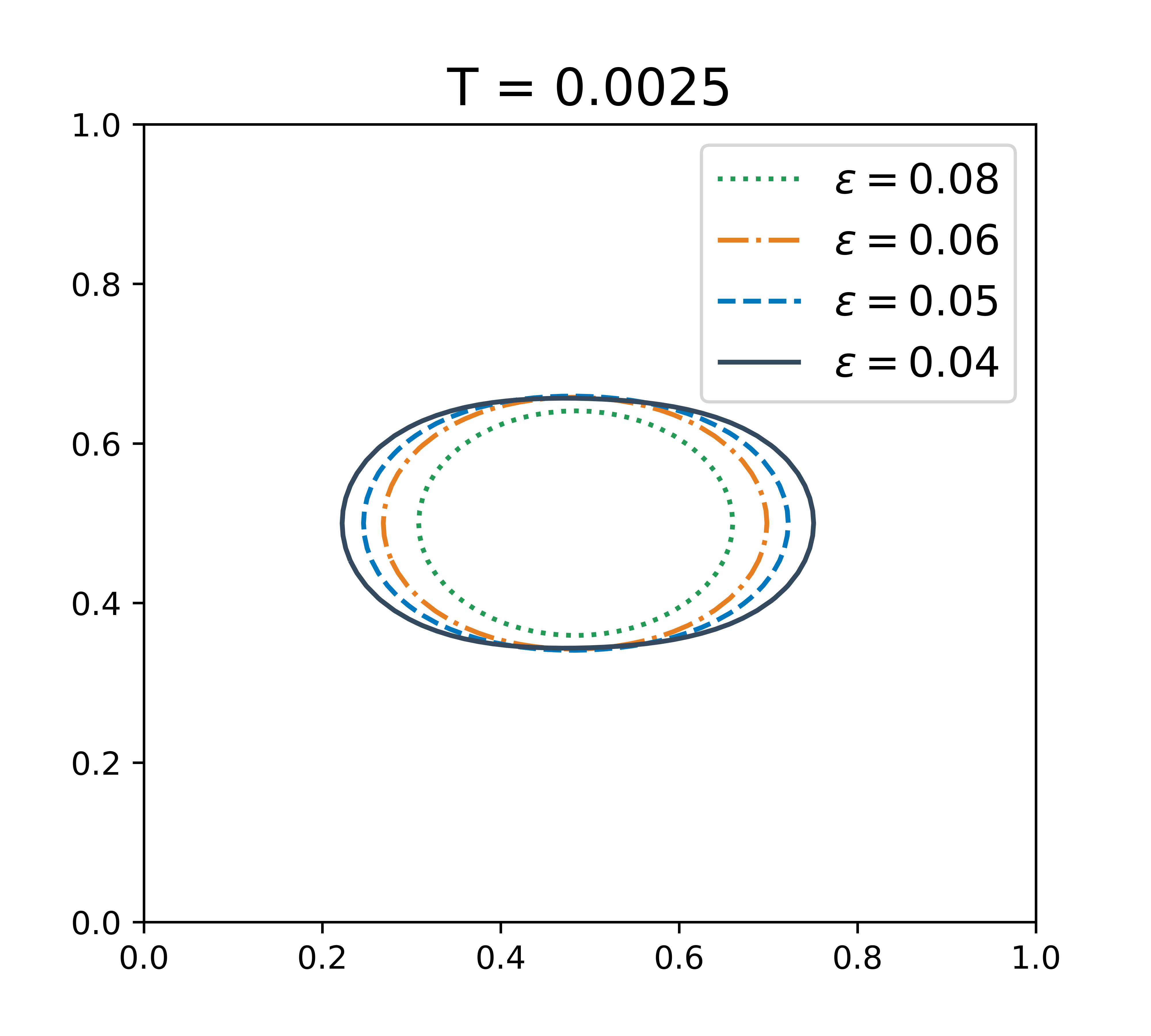}\hspace{-8pt}
\includegraphics[width=3.2in]{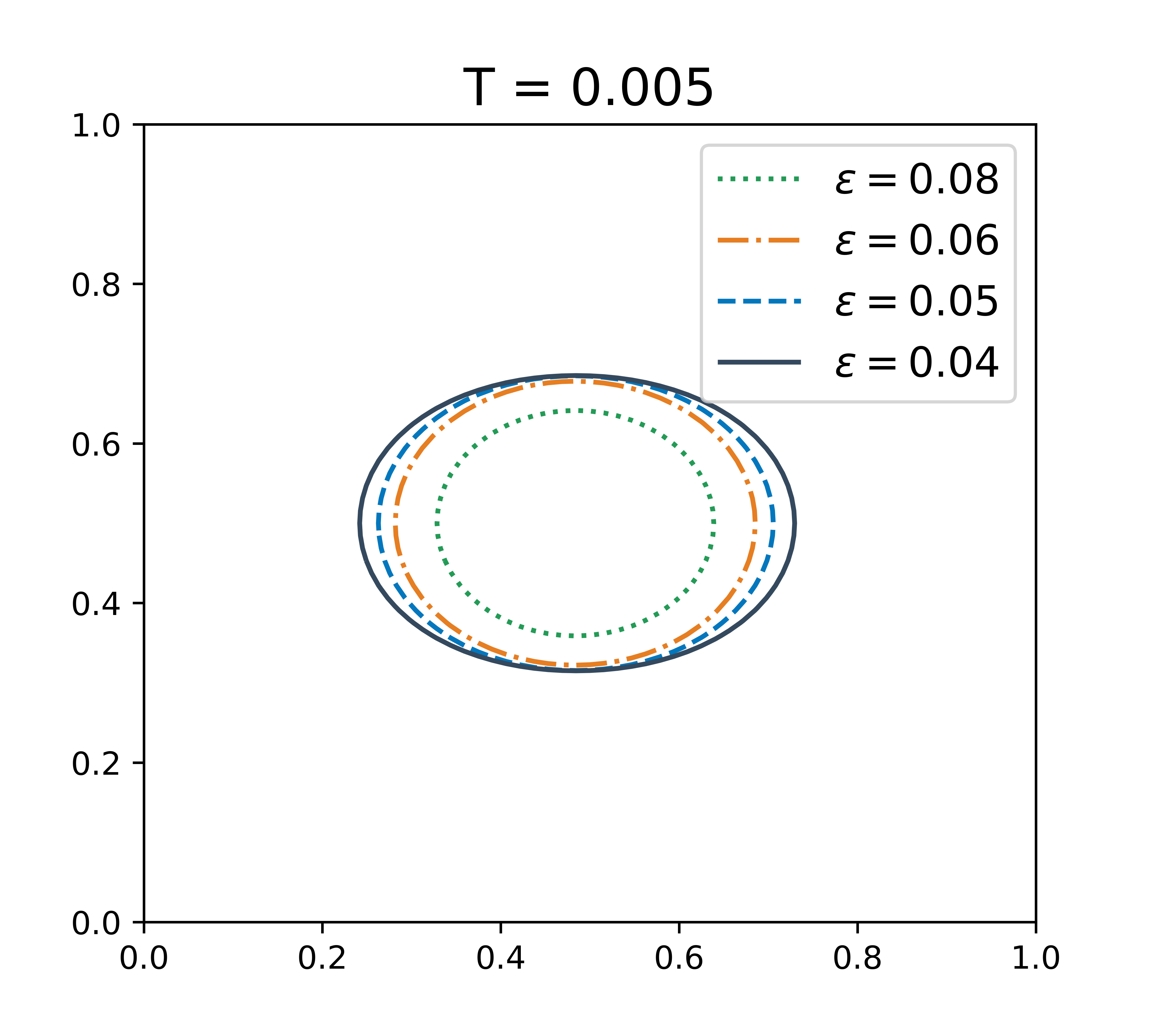}
}
\vspace{-8pt}
\caption{Snapshots of the zero-level sets of the numerical solutions.}
\label{fig_zero_level}
\end{figure}

\bigskip
\bigskip

\noindent \textit{Acknowledgements.} The work of Shu Ma was partially supported by the Research Grants Council of the Hong Kong Special Administrative Region, China. (Project Nos. CityU 11302718, CityU 11300621). 
The work of Weifeng Qiu was partially supported by the Research Grants Council of the Hong Kong Special Administrative Region, China. (Project Nos. CityU 11302718, CityU 11300621). All authors contribute equally. The second author is the corresponding author.

\bibliographystyle{abbrv}

\bibliography{CH-reference}

\end{document}